\documentclass[english,12pt]{amsart}
\usepackage{amsmath,amsfonts,amssymb,amsthm,fancyhdr,bm,verbatim,hyperref}
\usepackage{fullpage}
\usepackage{mathtools}
\usepackage{todonotes}

\makeatletter
\numberwithin{equation}{section}
\numberwithin{figure}{section}
\theoremstyle{plain}
\newtheorem{thm}{\protect\theoremname}
\theoremstyle{plain}
\newtheorem{cor}[thm]{\protect\corollaryname}
\theoremstyle{plain}
\newtheorem{lem}[thm]{\protect\lemmaname}
\newtheorem{conj}[thm]{\protect\conjecturename}
\newtheorem{prop}[thm]{\protect\propositionname}
\numberwithin{thm}{section}
\theoremstyle{definition}
\newtheorem*{Def}{Definition}

\newcommand{\phom}{p_{\textup{hom}}}
\newcommand{\F}{\mathcal{F}}
\newcommand{\E}{\mathbb{E}}

\makeatother

\newcommand{\Fpoly}{\F_{\textup{poly}}}
\newcommand{\Fpartpoly}{\F_{\textup{part-poly}}}
\newcommand{\Fpolyfact}{\F_{\textup{polyfact}}}
\newcommand{\Fpartpolyfact}{\F_{\textup{part-polyfact}}}

\usepackage{babel}
\providecommand{\corollaryname}{Corollary}
\providecommand{\lemmaname}{Lemma}
\providecommand{\theoremname}{Theorem}
\providecommand{\conjecturename}{Conjecture}
\providecommand{\propositionname}{Proposition}

\begin{document}

\subjclass[2020]{05C55, 05C65, 05D10, 05D40}
\title{Independent sets in hypergraphs with a forbidden link}
\author{Jacob Fox}
\thanks{Department of Mathematics, Stanford University, Stanford,
CA 94305, USA. Email: {\tt jacobfox@stanford.edu}. Research supported by
a Packard Fellowship and by NSF grant DMS-1855635.} 
\author{Xiaoyu He}
\thanks{Department of Mathematics, Stanford University, Stanford,
CA 94305, USA. Email: {\tt alkjash@stanford.edu}. Research supported by NSF Graduate Research Fellowship DGE-1656518.}
\maketitle

\begin{abstract}
We prove there exists a $3$-uniform hypergraph on $N$ vertices with independence number $O(\log N / \log \log N)$ in which there are at most two edges among any four vertices. This bound is tight and solves a longstanding open problem of Erd\H{o}s and Hajnal in Ramsey theory. The proof is based on a carefully designed probabilistic construction and the analysis depends on entropy methods. We further extend this result to prove tight bounds on various other hypergraph Ramsey numbers. 

\end{abstract}

\section{Introduction\label{sec:Introduction}}

Ramsey theory consists of many deep results across diverse fields of mathematics that can be summarized by the motto ``every large system contains a well-organized subsystem." Well-known examples include Dvoretzky's theorem in Banach space theory, the Paris-Harrington theorem in logic, the Hales-Jewett theorem in combinatorics, Szemer\'edi's theorem in number theory, and the equivalent Furstenberg multiple recurrence theorem in ergodic theory. 

The quintessential example, from which the area gets its name, is Ramsey's theorem. If $k\ge 2$, a $k$-uniform hypergraph, or simply $k$-graph, $G=(V,E)$ consists of a vertex set $V$ and an edge set $E\subseteq {V \choose k}$. Ramsey's theorem states that for any $k$-graphs $H_1$ and $H_2$, there is a positive integer $N$ such that any $k$-graph $G$ of order $N$ contains $H_1$ (as a subgraph) or its complement $\overline{G}$ contains $H_2$. The Ramsey number $r(H_1, H_2)$ is the smallest such $N$, and the main goal of graph Ramsey theory is to estimate $r(H_1,H_2)$, especially when one or both of $H_i$ is a complete graph $K_n$.

Even for the graph case $k=2$, the growth rate of Ramsey numbers is not well understood. The best known bounds on diagonal Ramsey numbers are of the form 
\[
\sqrt{2}^{(1+o(1))n}\le r(K_n,K_n) \le 4^{(1+o(1))n}
\]
and the exponential constants have not been improved since 1947; see~\cite{Co} and~\cite{Sp} for the best known bounds.

One of the crowning achievements of graph Ramsey theory was the determination of the order of the off-diagonal Ramsey number
\begin{equation}\label{eq:triangle}
r(K_3, K_n) = \Theta\Big(\frac{n^2}{\log n}\Big).
\end{equation}
The lower bound was proved first by Ajtai, Koml\'os, and Szemer\'edi~\cite{AjKoSz}, and the matching upper bound was proved by Kim~\cite{Ki}. An impressive amount of progress (see~\cite{BoKe, FiGrMo, Sh}) has been made towards finding the exact asymptotic in (\ref{eq:triangle}), leaving only a multiplicative gap of $4+o(1)$ between the best known bounds; see Spencer \cite{Sp11} for a history of this problem and its influence on the development of the probabilistic method. In stark contrast, the order (and even the exponent) of $r(G, K_n)$ is still open for almost all other graphs $G$, including all cycles or complete graphs on at least $4$ vertices.

We know even less about hypergraph Ramsey numbers of uniformity $k\ge 3$. Let $K_n^{(k)}$ be the complete $k$-graph on $n$ vertices. In the diagonal case (see \cite{CoFoSu,CoFoSu2,CoFoSu3,GrRoSp}), the best known bounds for $k=3$,
\begin{equation}\label{eq:diagonal-hypergraph-ramsey}
2^{\Omega(n^2)} \le r(K_n^{(3)}, K_n^{(3)}) \le 2^{2^{O(n)}},
\end{equation}
differ by an entire exponential order. For each larger uniformity greater than $3$, the lower and upper bounds go up by one exponential order, leaving the same type of exponential gap. Techniques developed by Erd\H{o}s-Rado and Erd\H{o}s-Hajnal allow us to lift bounds on $3$-graph Ramsey numbers to all higher uniformities, so closing the gap in (\ref{eq:diagonal-hypergraph-ramsey}) would also close the gap for all $k\ge 3$. Thus, the case $k=3$ is of central importance in this subject.

There is also a large gap between the lower and upper bounds for off-diagonal hypergraph Ramsey numbers.  For $s \geq 4$ fixed, the best known bounds \cite{CoFoSu} are of the form
\begin{equation}\label{offhypthree}
2^{\Omega(sn \log (2n/s))} \le r(K_s^{(3)}, K_n^{(3)}) \le 2^{O(n^{s-2}\log n)}.
\end{equation}
In fact, the lower bound holds for all $s\le n$. In this paper, we will study a different, but closely related, off-diagonal hypergraph Ramsey number.

In 1972, Erd\H{o}s and Hajnal \cite{ErHa} posed the problem of determining the minimum independence number of a $3$-graph on $N$ vertices in which there are at most two edges among any four vertices, and showed that the answer is between $\Omega(\log N/\log \log N)$ and $O(\log N)$. In the language of Ramsey numbers, this is equivalent to 
\begin{equation}\label{K43minusedge}
    2^{\Omega(n)}\le r(K_4^{(3)} \setminus e, K_n^{(3)}) \le n^{O(n)},
\end{equation}
where $K_4^{(3)}\setminus e$ is the $3$-graph on four vertices and three edges. This problem has received considerable attention during the half-century since it was posed, see \cite{CoFoSu,Er90,MuSu18,MuSu19}. Mubayi and Suk \cite{MuSu18} wrote that it ``is a very interesting open problem, as $K_4^{(3)} \setminus e$ is, in some sense, the smallest 3-uniform hypergraph whose Ramsey number with a clique is at least
exponential.''

We solve this problem of Erd\H{o}s and Hajnal and show that the upper bound in (\ref{K43minusedge}) is tight. 
\begin{thm}\label{thm:K43minusedge}
We have $r(K_4^{(3)} \setminus e, K_n^{(3)})=n^{\Theta(n)}$.
\end{thm}

We prove Theorem~\ref{thm:K43minusedge} with a carefully designed probabilistic construction, and entropy inequalities are used in the analysis.

In the original language of Erd\H{o}s and Hajnal, Theorem~\ref{thm:K43minusedge} states that the minimum possible independence number of a $(K_4^{(3)}\setminus e)$-free $3$-graph on $N$ vertices is $\Theta(\log N/\log \log N)$. We further generalize this result to Theorem~\ref{thm:linkhypergraph} below, exhibiting a large family of $3$-graphs $H$ which satisfy $r(H,K_n^{(3)})=n^{\Theta(n)}$.

Given a graph $G$, its {\it link hypergraph} $L_{G}$ is the $3$-graph with vertex set $V(G)\cup\{u\}$ whose edges are exactly the triples $\{u,v,w\}$ for which $\{v,w\}\in E(G)$. Observe that $K_4^{(3)} \setminus e$ is $L_{K_3}$, the link hypergraph of the triangle. Conlon, Fox, and Sudakov \cite{CoFoSu} proved that if $G$ is bipartite, then $r(L_G,K_{n}^{(3)})=n^{\Theta(1)}$, where the implied constants depend on $G$, and if $G$ is non-bipartite, then $r(L_G,K_{n}^{(3)})\ge 2^{\Omega(n)}$. We improve this lower bound.

\begin{thm}\label{thm:linkhypergraph}
If $G$ is a fixed non-bipartite graph, then $r(L_G,K_n^{(3)}) = n^{\Theta(n)}$. 
\end{thm}

The implied constants in Theorem \ref{thm:linkhypergraph} depend on $G$. Theorem \ref{thm:linkhypergraph} determines the behavior of off-diagonal Ramsey numbers for link hypergraphs and implies the curious fact that they always grow either polynomially or superexponentially in $n$. 

To prove the lower bound in Theorem \ref{thm:linkhypergraph}, we will prove a stronger result. To state it, we need a few definitions. The {\it link} $F_v$ of a vertex $v$ in a $k$-graph $F$ is the $(k-1)$-graph on $V(F) \setminus \{v\}$ where $e$ is an edge of $F_v$ if $e \cup \{v\} \in E(F)$. Note that a $3$-graph is $L_G$-free if and only if the link of each vertex is $G$-free. The {\it odd girth} of a non-bipartite graph is the length of its shortest odd cycle, and is infinite for a bipartite graph. Finally, let $K_{n,n,n}^{(3)}$ be the complete balanced tripartite $3$-graph on $3n$ vertices for which the edges are precisely the triples with one vertex in each part.

\begin{thm}\label{thm:oddgirthlowerbound}
For all $g\ge 3$ and $n$ sufficiently large in terms of $g$, there is a $3$-graph $\Gamma$ on  $n^{\Omega(n/g)}$ vertices such that the link of each vertex has odd girth greater than $g$ and the complement of $\Gamma$ is $K_{n,n,n}^{(3)}$-free. 
\end{thm}

One can think of a graph with large odd girth as being ``locally bipartite." Theorem~\ref{thm:oddgirthlowerbound} shows there exists a $3$-graph on $n^{\Omega(n)}$ vertices with independence number $n$ that is ``locally bipartite" in this sense. In contrast, it can be proved by induction on $n$ that every $3$-graph on $2^{n-1}$ vertices in which the link of each vertex is bipartite has independence number at least $n$. Thus, Theorem~\ref{thm:oddgirthlowerbound} implies that it is substantially easier for a $3$-graph to have all links ``locally bipartite" than to have all links bipartite. 

The proof of Theorem~\ref{thm:oddgirthlowerbound} constructs the $3$-graph $\Gamma$ randomly. The link of each vertex is chosen to be a subgraph of a random blowup of some predetermined auxiliary graph with large odd girth; note that blowing up a graph does not decrease its odd girth. To show that the complement of $\Gamma$ is $K_{n,n,n}^{(3)}$-free with positive probability, we use a counting argument motivated by entropy methods.

To prove the upper bound in Theorem~\ref{thm:linkhypergraph}, we use the so-called ``vertex online Ramsey game" of Conlon, Fox, and Sudakov \cite{CoFoSu}, which has been used to prove upper bounds for many hypergraph Ramsey numbers. The explicit upper bound we prove is as follows.

\begin{thm}\label{thm:linkcliqueversusclique}
For all $s,n \geq 3$, we have $r(L_{K_s},K_n^{(3)})<(2n)^{sn}$. 
\end{thm}

Notice that Theorems~\ref{thm:oddgirthlowerbound} and~\ref{thm:linkcliqueversusclique} together show that if $G$ is a graph on $s$ vertices with odd girth $g\ge 3$, 
then for $n$ sufficiently large we have 
\[
\Omega(1/g) \le \frac{\log r(L_{G},K_n^{(3)})}{n \log n} \le O(s).
\]
It would be interesting to determine the exact dependence on $G$.

We now turn to the Ramsey numbers of link hypergraphs of cliques against cliques, so that the link hypergraph is no longer fixed in size. Our main result in this direction implies that for any fixed $\varepsilon>0$ and $s<n^{1-\varepsilon}$, Theorem \ref{thm:linkcliqueversusclique} is tight up to a constant factor in the exponent.

\begin{thm}\label{thm:linkcliquelower}
For all $s \geq 3$ and $n\ge 1$, we have  \[r(L_{K_s},K^{(3)}_{n,n,n}) = {n+s \choose s}^{\Theta(n)},\] where the implicit constant factors in the exponent do not depend on $s$. \end{thm}

Note that Theorem \ref{thm:linkcliquelower} is equivalent to
\[
    r(L_{K_s},K^{(3)}_{n,n,n})= 
\begin{cases}
     \left(\frac{2n}{s}\right)^{\Theta(sn)}& \text{if } s\leq n,\\
    \left(\frac{2s}{n}\right)^{\Theta(n^2)}              & \text{if $s>n$.}
\end{cases}
\]
Furthermore, our proof shows that Theorem \ref{thm:linkcliquelower} also holds for $s\ge 14$ if we replace $K^{(3)}_{n,n,n}$ by any dense hypergraph on $\Theta(n)$ vertices which is a sub-hypergraph of a blow-up of $L_{K_n}$. 

Erd\H{o}s and Hajnal \cite{ErHa} introduced and began the study of the following function in order to get better insight into hypergraph Ramsey numbers. 

\begin{Def}
Let $f_k(N;s,t)$ denote the minimum possible independence number of a $k$-graph on $N$ vertices in which any $s$ vertices contain fewer than $t$ edges.  
\end{Def}

For example, $f_3(N,4,3)$ is essentially the inverse function of $r(K_4^{(3)} \setminus e,K_n^{(3)})$. 

For fixed $k\ge 3$ and $s \ge k+1$, consider the rate of growth of the functions $f_k(N;s,t)$ as $t$ increases from $1$ to ${s\choose k}$. Erd\H{o}s and Hajnal  conjectured that starting from $t=1$, $f_k(N;s,t)$ grows like a power of $N$, then after a point $t = h_1^{(k)}(s)$, it grows like a power of $\log N$, and then there's a point $t = h_2^{(k)}(s)$ at which it grows like a power of $\log \log N$, and so on up to a threshold $t = h_{k-2}^{(k)}(s)$ past which it grows like the $(k-2)$-fold iterated logarithm $\log _{(k-2)} N$. 

Erd\H{o}s and Hajnal further conjectured a specific value for $h_1^{(k)}(s)$, the polynomial to exponential threshold, and Erd\H{o}s later (see \cite{ChGr}) offered \$500 to resolve this conjecture. This conjecture was solved for $k=3$ and infinitely many $s$ by Conlon, Fox, and Sudakov \cite{CoFoSu} and recently solved for $k \geq 4$ by Mubayi and Razborov \cite{MuRa}. Mubayi and Suk \cite{MuSu19+} recently proved the general conjecture in the case $s=k+1$.

Despite this progress on the rough behaviour, the order of $f_k(N;s,t)$ has remained wide open. Using our methods, we determine the order of $f_3(N;s,t)$ for a constant fraction of values of $t$ for each fixed $s$.

\begin{thm}\label{thm:constant-fraction}
There exist functions $e(s) \le (0.266+o(1)){s \choose 3}$ and $E(s) \ge  (0.464 - o(1)){s \choose 3}$ such that for any fixed $s\ge 4$ and $t\in (e(s), E(s)]$, 
\[
f_3(N;s,t) = \Theta(\log N/\log \log N),
\]
where the implicit constants are allowed to depend on $s$.
\end{thm}

The function $e(s)$ is the maximum number of edges in a $3$-graph on $s$ vertices in which the link of each vertex is bipartite. The function $E(s)$ is the maximum number of edges in a $3$-graph $H$ on $s$ vertices for which $r(H,K_n^{(3)})=n^{O(n)}$. The bound on $E(s)$ comes from a recursive construction of $3$-graphs with this property. 

\noindent {\bf Organization.} We prove the lower bounds and upper bounds on Ramsey numbers separately. In Section~\ref{sec:setup}, we present the randomized construction that is used for all the lower bound proofs. Section~\ref{sec:large-d} uses this machinery to prove the lower bound in Theorem~\ref{thm:linkcliquelower}, while Section~\ref{sec:main} uses it to prove Theorem~\ref{thm:oddgirthlowerbound}. 

After that, we prove the upper bound in Theorem~\ref{thm:linkcliqueversusclique} in Section~\ref{sec:link-upper-bound} using the vertex online Ramsey game. Finally, we prove the upper bound in Theorem~\ref{thm:linkcliquelower} in Section~\ref{sec:recursive}, and study the more general problem of classifying all $3$-graphs $H$ for which $r(H,K_n^{(3)})=n^{O(n)}$, proving Theorem~\ref{thm:constant-fraction}. We conclude with Section~\ref{sec:closing} describing some related open problems on hypergraph Ramsey numbers.

For the sake of clarity of presentation, we systematically omit floor and ceiling signs whenever they are not crucial.

\section{The Lower Bound Construction}\label{sec:setup}

In this section, we sketch the lower bound proofs, as the construction and some of the details of the proof are the same for the lower bounds in Theorems~\ref{thm:oddgirthlowerbound} and~\ref{thm:linkcliquelower}. Given a non-bipartite graph $G$ and a dense $3$-graph $F$,\footnote{In both Theorems~\ref{thm:oddgirthlowerbound} and~\ref{thm:linkcliquelower}, we take $F=K_{n,n,n}^{(3)}$, although we mainly use that $F$ is dense.} we will construct a large $L_G$-free $3$-graph $\Gamma$, such that $\overline{\Gamma}$ contains no copy of $F$. 

We say that a graph $A$ is $\hom(G)$-free if there does not exist a graph homomorphism from $G$ to $A$. The first step is to construct an auxiliary $\hom(G)$-free graph $A$ that will look like the Erd\H os-R\'enyi random graph $G(m,p)$ where $m\ge 1$ and $p\in (0,1)$ are chosen appropriately. 

Next, we pick a uniform random function $\chi: [N]^2 \rightarrow V(A)$. We simply ignore the values $\chi(v,v)$ along the diagonal, so $\chi$ can be interpreted as a random edge-coloring of the complete directed graph $K_N^*$ (with edges in both directions for each distinct pair of vertices) by elements of $V(A)$. We call a triple $1\le i<j<k\le N$ an {\it adjacent triangle} of $\chi$ if and only if the three adjacencies $\chi(i,j)\sim\chi(i,k)$, $\chi(j,i)\sim\chi(j,k)$,
and $\chi(k,i)\sim\chi(k,j)$ hold simultaneously in $A$. Then the final $\Gamma=\Gamma(\chi)$ is the $3$-graph on vertex set $[N]$ whose edges are the adjacent triangles of $\chi$.

The next lemma motivates the definition of $\Gamma(\chi)$ and our choice of $A$. The upshot is that the link of every vertex in $\Gamma(\chi)$ is some subgraph of a blowup of $A$.

\begin{lem}\label{lem:surj-free}
If $A$ is $\hom(G)$-free, then $\Gamma(\chi)$ is $L_G$-free.
\end{lem}
\begin{proof}
Suppose $\Gamma(\chi)$ contains a copy of $L_G$, so some vertex $v\in [N]$ contains a copy of $G$ in its link. Let the vertices of this copy be $u_1,\ldots, u_s$. By the definition of $\Gamma(\chi)$, we know that $\chi(v, u_i)\sim \chi(v,u_j)$ if $u_i$ and $u_j$ represent adjacent vertices in $G$. Equivalently, we have shown that the map $\phi:\{u_1,\ldots, u_s\} \rightarrow V(A)$ which takes values $\phi(u_i)=\chi(v,u_i)$ is a graph homomorphism from $G$ to $A$. This contradicts our assumption that $A$ is $\hom(G)$-free.
\end{proof}

It remains to prove that with positive probability, $\overline{\Gamma(\chi)}$ contains no copy of $F$. We do this by proving that the probability a given copy of $F$ appears in $\overline{\Gamma(\chi)}$ is extremely small, and then taking the union bound over all possible copies.

It is worth comparing our construction $\Gamma(\chi)$ to the earlier construction of Conlon, Fox, and Sudakov~\cite{CoFoSu}. In our language, their construction is equivalent to considering the $3$-graph $\Gamma'(\chi)$ in which a triple $i<j<k$ is an edge if and only if the single condition $\chi(i,j)\sim\chi(i,k)$ holds. Unlike $\Gamma(\chi)$, $\Gamma'(\chi)$ has the property that the appearances of the edges whose leftmost vertex is $i$ are independent from the appearances of the edges whose leftmost vertex is $j$, whenever $i\ne j$. Thus, by this independence, the probability that $\overline{\Gamma'(\chi)}$ has a large clique can be expanded into a product over one leftmost vertex at a time, and the argument is straightforward from there. Because of the failure of this independence in our construction $\Gamma(\chi)$, we will need to use more delicate arguments to achieve the same type of bounds.

The lower bound proofs are organized as follows. In Section~\ref{sec:large-d}, we show that the lower bound in Theorem~\ref{thm:linkcliquelower} is true when $s\ge 14$. Then, in Section~\ref{sec:main}, we show Theorem~\ref{thm:oddgirthlowerbound}, which also proves the lower bound in Theorem~\ref{thm:linkcliquelower} for absolutely bounded $s$. Although both proofs use the machinery above, the former argument is much easier because $A$ is chosen to be a random graph, so in the analysis of $\Gamma(\chi)$ there are two independent sources of randomness: that of $A$ and that of $\chi$. For the latter proof, we will fix an $A$ and use entropy arguments instead of independence.

\section{The lower bound in Theorem~\ref{thm:linkcliquelower}}\label{sec:large-d}

The lower bound in Theorem \ref{thm:linkcliquelower} for $s\geq 14$ follows from Lemma~\ref{lem:lower-chromatic} below, which gives a bound on the Ramsey number of the link hypergraph $L_G$ of a graph $G$ of sufficiently large chromatic number versus any general dense $3$-graph $F$. Moreover, while the probabilistic construction is based on the same idea as that used to prove Theorem \ref{thm:oddgirthlowerbound}, the proof of the lower bound in Lemma~\ref{lem:lower-chromatic} is simpler. 

For $m\ge 2$, define $\phom(G, m)$ to be the largest value of $p\in [0,1]$ for which the Erd\H os-R\'enyi random graph $A=G(m,p)$ is $\hom(G)$-free with probability at least $\frac 12$. 

\begin{lem}\label{lem:lower-chromatic}
If $G$ is a nonempty graph, $m\ge 2$, and $p \le \phom(G,m)$, then for any $3$-graph $F$ on $n\ge 3$ vertices with  $\delta n^3$ edges,
\[
r(L_G, F) \ge \frac{1}{2} \min\Big\{(1-p^3)^{-n}, \frac{\delta m}{2e n^2}\Big\}^{\delta n/2}.
\]
\end{lem}

As in Section~\ref{sec:setup}, $\chi:[N]^2\rightarrow V(A)$ is chosen uniformly at random among all possible such functions, and $\Gamma(\chi)$ is the $3$-graph of adjacent triangles of $\chi$. We first bound the probability a particular copy of $F$ appears in $\overline{\Gamma(\chi)}$. Suppose the vertex set of $F$ is $[n]$, and $U = (u_1, \ldots, u_n)$ is a sequence of distinct vertices in $[N]$. Let $F_{U}$ be the event that the map $i\mapsto u_i$ is an embedding of $F$ into $\overline{\Gamma(\chi)}$. 

\begin{lem}\label{lem:F-free-chromatic}
Suppose $F$ is a $3$-graph on $n$ vertices with $\delta n^3$ edges and $A=G(m,p)$. For each sequence $U\in [N]^n$ of pairwise distinct vertices,
\[
\Pr[F_U] \le (1-p^3)^{\delta n^3/2} + \Big(\frac{2en^2}{\delta m}\Big)^{\delta n^2/2}. 
\]
\end{lem}
\begin{proof}
Without loss of generality, $U = (1,\ldots, n)$ and we may identify $U$ with the vertex set of $F$. The relevant values of $\chi$ are $\chi(i,j)$ for all $n(n-1)$ ordered pairs $(i,j)\in [n]^2$ with $i\ne j$. Imagine that the values of $\chi$ are randomly and uniformly picked sequentially in lexicographic order of $(i,j)$.

We first consider the event $E$ that at most $(1-\frac{\delta}{2}) n^2$ vertices of $A$ appear as colors in $\chi|_U$. If $E$ occurs, there is a set $R$ of $\frac{\delta}{2} n^2$ ``repeated edges" $(i,j)$ such that $\chi(i, j)$ has already appeared among the pairs colored earlier. There are at most ${n^2 \choose \frac{\delta}{2} n^2} < (2e/\delta)^{\delta n^2/2}$ ways of choosing such a set $R$. We will fix $R$ and bound the probability of the event $E_R$ that every edge in $R$ is a repeated edge.

For each $(i, j)\in R$, at most $n^2$ colors have already been used from $V(A)$, so the probability that $\chi(i,j)$ is one of these colors is at most $n^2/m$. Hence, 
\[
\Pr[E] < (2e/\delta)^{\delta n^2/2} \max_R\Pr[E_R] \le (2e/\delta)^{\delta n^2/2} (n^2/m)^{\delta n^2/2} = \Big(\frac{2en^2}{\delta m}\Big)^{\delta n^2/2}. 
\]

Next, we condition on $E$ not occurring and let $R$ be the set of repeated pairs, so $|R|\le \delta n^2/2$. Since each repeated pair $(i,j)$ lies in at most $n$ edges of $F$ and $F$ has $\delta n^3$ total edges, at least $\delta n^3/2$ edges of $F$ do not contain any repeated pair. Let $T_R$ denote the set of all such edges, so that $T_R$ is a subset of ${[n] \choose 3}$ of cardinality at least $\delta n^3/2$.

If $t=\{i,j,k\}\in {[n] \choose 3}$, let $X_t$ denote the event that $\{ i,j,k \}$ is an adjacent triangle in $\chi$. The key observation is that all the events $X_t$ for $t\in T_R$ are mutually independent, and $\Pr[X_t]=p^3$ for all such $t$. Indeed, observe that if $t=\{i, j, k\}$, then $X_t$ is just the conjunction
\[
X_t = (\chi(i, j)\sim \chi(i, k)) \wedge (\chi(j, i)\sim \chi(j, k)) \wedge  (\chi(k, i)\sim \chi(k, j)).
\]
Each of the three events above is the event that a particular edge appears in $A=G(m,p)$, and all six colors in $V(A)$ that appear are distinct since $t\in T_R$, so these three events refer to three distinct edges. Thus, $\Pr[X_t]=p^3$. Furthermore, all the $6|T_R|$ colors that appear on edges in some triple $t\in T_R$ must be distinct, since there are no repeated edges in triples of $T_R$. Thus, the events $X_t$ depend on pairwise disjoint triples of edges appearing in $G(m,p)$, so they are all mutually independent. Hence, 
\[
\Pr[F_U | \neg E] \le \Pr\Big[\bigwedge_{t\in T_R} (\neg X_{t}) \Big| |R| \le \delta n^2/2\Big] \le (1-p^3)^{\delta n^3/2}.
\]
In total, we have
\[
\Pr[F_U] \le \Pr[F_U | \neg E] + \Pr[E] \le (1-p^3)^{\delta n^3/2} + \Big(\frac{2en^2}{\delta m}\Big)^{\delta n^2/2},
\]
as desired.
\end{proof}

\noindent We can now prove the main lemma.

\begin{proof}[Proof of Lemma~\ref{lem:lower-chromatic}.]
We are now working under the assumptions of Lemma~\ref{lem:lower-chromatic}. That is, $m\ge 2$, $G$ is a graph on $m$  vertices, $p=\phom(G,m)$, $n \geq 3$, $F$ is a $3$-graph on $n$ vertices with at least $\delta n^3$ edges, and
\begin{equation}\label{eq:M-def}
M = \min\Big((1-p^3)^{-n}, \frac{\delta m}{2e n^2}\Big).
\end{equation}
Let $N = \frac{1}{2} M^{\delta n/2}$.

The setup is the same as in Section~\ref{sec:setup}. With $A=G(m,p)$, pick a uniform random coloring $\chi:[N]^2\rightarrow V(A)$, and define $\Gamma(\chi)$ to be the $3$-graph on $[N]$ whose edges are the adjacent triangles of $\chi$. By Lemma~\ref{lem:surj-free} and the definition of $\phom(G,m)$, the probability that $\Gamma(\chi)$ contains a copy of $L_G$ is at most $\frac{1}{2}$.

As before, we defined $F_U$ to be the indicator random variable of the event that a given ordered copy of $F$ appears in $\overline{\Gamma(\chi)}$. By Lemma~\ref{lem:F-free-chromatic}, 
\[
\Pr[F_{U}] \le (1-p^3)^{\delta n^3/2} + (2en^2/\delta m)^{\delta n^2/2}.
\]
Comparing with equation (\ref{eq:M-def}), this implies
\[
\Pr[F_U] \le 2M^{-\delta n^2/2}.
\]

There are less than $N^n$ total choices of $U$, so by the union bound the probability that $F$ appears in $\overline{\Gamma(\chi)}$ is less than
\[
N^n \cdot 2M^{-\delta n^2/2} \le \frac{1}{2}.
\]

This shows that with positive probability $\Gamma(\chi)$ is an $L_G$-free $3$-graph and $\overline{\Gamma(\chi)}$ contains no copies of $F$, as desired.
\end{proof}

To finish the proof of the lower bound in Theorem~\ref{thm:linkcliquelower}, it suffices to apply known results on the value of $\phom(G,m)$.

\begin{proof}[Proof of the lower bound in Theorem~\ref{thm:linkcliquelower}.]
If $s\le 13$, apply Theorem~\ref{thm:oddgirthlowerbound} (which we prove in the next section) with $g=3$ to obtain the desired lower bound. So we may assume $s\ge 14$.

For a clique $K_s$, being $\hom(K_s)$-free is equivalent to being $K_s$-free. Thus, the threshold $\phom(K_s,m)$ is just the threshold for containing an $s$-clique, one of the first threshold functions determined by the seminal work of Erd\H{o}s and R\'enyi~\cite{ErRe}. For our purposes, it is sufficient to use $\phom(K_s, m) \ge m^{-\frac{2}{s-1}}$, which follows from the union bound. 

We will apply Lemma~\ref{lem:lower-chromatic} with $F = K_{n,n,n}^{(3)}$ on $3n$ vertices and $\delta (3n)^3$ edges, where $\delta = \frac{1}{27}$. This gives
\[
r(L_{K_s}, K_{n,n,n}^{(3)}) \ge \frac{1}{2} \min \Big\{ (1-p^3)^{-3n}, \frac{m}{486en^2}\Big\}^{n/18},
\]
for any $m\ge 2$ and $p = m^{-\frac{2}{s-1}}$. Thus, it suffices to show that there exists $m$ for which
\[
    \min \Big\{ (1-p^3)^{-3n}, \frac{m}{486en^2}\Big\} \ge {n + s \choose s}^{\Omega(1)}.
\]

We prove this holds with $m={n + s \choose s}^{2/13}$. The calculation is carried out in Appendix~\ref{app:threshold}.
\end{proof}

\section{Proof of Theorem~\ref{thm:oddgirthlowerbound}}\label{sec:main}
We come to the most involved argument in this paper, which is the proof of Theorem~\ref{thm:oddgirthlowerbound}. 

For a $k$-graph $G$, we call a copy of the complete balanced $k$-partite $k$-graph $K_{t,\ldots, t}^{(k)}$ in $\overline{G}$ a {\it $k$-partite independent set} of $G$ of order $t$, and let $\alpha_k(G)$ be the largest $t$ such that $G$ has a $k$-partite independent set of order $t$.

Given $g\ge 3$ and $n$ large, the goal of this section is to give a probabilistic construction of a $3$-graph $\Gamma$ on $N=n^{\Omega(n/g)}$ vertices such that with positive probability the link of every vertex in $\Gamma$ has odd girth greater than $g$, and $\alpha_3(\overline{\Gamma}) < n$.

\subsection{The auxiliary graph $A$}\label{subsec:A-construction}
In this section we construct the auxiliary graph $A$ which has  girth greater than $g$ and $\alpha_2(A)$ small. Erd\H os showed that a graph with high girth and relatively small independence number can be obtained by deleting
vertices from all short cycles in an appropriate Erd\H os-R\'enyi
random graph. With a slight modification of his argument, this can be done to make $\alpha_{2}(A)$ small as well.
\begin{lem}
\label{lem:A-construction}For any $g\ge3$ and $m$ sufficiently
large in terms of $g$, there exists a graph $A$ of order $m$ with girth greater
than $g$ such that
\[
\alpha_{2}(A)<\frac{1}{16}m^{1-\frac{1}{3g}}.
\]
\end{lem}

\begin{proof}
We sample a random graph $G(2m,p)$ with $p=(4m)^{-1+\frac{1}{g}}$, and then delete a vertex from each cycle of length at most $g$ and from
each bipartite independent set of order $t=\frac{1}{16}m^{1-\frac{1}{3g}}$.
Let us compute the expected number of vertices deleted.

By linearity of expectation, in $G(2m,p)$ we have
\[
\mathbb{E}[\#\text{ cycles of length \ensuremath{\ell}}]\le\frac{(2m)^{\ell}}{2\ell}\cdot p^{\ell},
\]
since there are at most $(2m)^{\ell}/2\ell$ ways to choose a cycle
of length $\ell$ to appear, and each cycle of length $\ell$ appears
with probability $p^{\ell}$. It follows that
\[
\mathbb{E}[\#\text{ cycles of length at most \ensuremath{g}}]\le\sum_{\ell=3}^{g}\frac{(2m)^{\ell}}{2\ell}\cdot p^{\ell}\le \sum_{\ell=3}^{g}2^{-\ell}\ell^{-1}(2m)^{\ell/g}\le\frac{m}{2}.
\]
We thus expect to delete at most $m/2$ vertices to destroy all cycles of length at most $g$.

The total number of choices of disjoint sets $(U,W)$ with $|U|=|W|=t$ is at most
$\binom{m}{t}^2 \le 4^{m}$. For each such pair, the probability it induces a bipartite independent set
is $(1-p)^{t^2}$. Thus, by linearity of
expectation again,
\[
\mathbb{E}[\#\text{ bipartite independent sets of order \ensuremath{t}}]\le4^{m}(1-p)^{t^2}.
\]
Also, $1-p\le e^{-p}$ and $t^2=\frac{1}{256}m^{2-\frac{2}{3g}}$, so 
\[
4^{m}(1-p)^{t^2}\le4^{m}e^{-\frac{1}{256}pm^{2-\frac{2}{3g}}}\le4^{m}e^{-\frac{1}{1024}m^{1+\frac{1}{3g}}},
\]
which tends to zero as $m$ tends to infinity. For sufficiently large $m$, we need to delete at most $m$ vertices on average from $G(2m,p)$ to obtain a graph
with girth greater than $g$ and $\alpha_{2}(A)<\frac{1}{16}m^{1-\frac{1}{3g}}$. In particular, with positive probability we delete at most $m$ vertices,
so by deleting additional vertices if necessary we obtain a graph $A$ of order $m$.
\end{proof}

We will actually require an upper bound on the bipartite independence number of the tensor square of $A$, denoted $A^{2}$, which is the graph on $V(A)^{2}$ where
$(u_{1},u_{2})\sim(v_{1},v_{2})$ in $A^{2}$ if and only if $u_{1}\sim v_{1}$ and $u_{2}\sim v_{2}$ in $A$. There is a simple counting argument which relates $\alpha_2(A^{2})$ to $\alpha_2(A)$.

\begin{lem}\label{lem:tensor-square}
For any graph $A$ of order $m$, $\alpha_2(A^{2}) < 4 m(\alpha_2(A)+1)$.
\end{lem}
\begin{proof}
Suppose
that $A^{2}$ contains a large bipartite independent set induced on two parts $U,W\subseteq V(A)^2$
of order $4m(\alpha_2(A)+1)$ each. Let $U_{1}$ (respectively $W_{1}$) be the
set of vertices in $A$ which appear at least $2(\alpha_2(A)+1)$
times as the first coordinate of a vertex in $U$ (respectively $W$). Since
each $u_{1}\not\in U_{1}$ appears less than $2(\alpha_2(A)+1)$
times as a first coordinate in $U$, there are at most $2m(\alpha_2(A)+1)$
pairs $(u_{1},u_{2})\in U$ where $u_{1}\not\in U_{1}$. That is,
at least half of the vertices of $U$ have one of the ``popular'' first coordinates
$U_{1}$, and so $|U_{1}|\ge|U|/2m\ge2(\alpha_2(A)+1)$. Similarly, $|W_{1}|\ge2(\alpha_2(A)+1)$. 

From these sets $U_1, W_1$, it is possible to pick disjoint subsets $U_1 ' \subset U_1$, $W_1 ' \subset W_1$ of size $|U_1'|=|W_1'|=\alpha_2(A)+1$. By the definition of $\alpha_2(A)$, there must be an edge $(u_{1},w_{1})$ between these sets. Fix these two vertices $u_1, w_1$.

Next, define $U_{2}$ to be the set of all $u_{2}$ for which $(u_{1},u_{2})\in U$, and define $W_2$ similarly for $W$.
By the definitions of $U_{1}$ and $W_{1}$, we have $|U_{2}|,|W_{2}|\ge2(\alpha_2(A)+1)$. Again, we can find disjoint subsets $U_2 ' \subset U_2$, $W_2 ' \subset W_2$ of size $|U_2'|=|W_2'|=\alpha_2(A)+1$, whereby there must be an edge $(u_2,w_2)$ between $U_2$ and $W_2$ as well.

It follows that $(u_{1},u_{2})\sim(w_{1},w_{2})$ is an edge between $U$
and $W$ in $A^{ 2}$, which contradicts the fact that $(U,W)$ induces a bipartite independent set.
\end{proof}

For convenience, we make an observation about $\alpha_2(A)$. 
When $U\subseteq V(A)$, define $\overline N(U)$ to be the set of all vertices of $A$ with no edges to vertices of $U$, so that $\overline N(U)$ is allowed to include vertices of $U$ itself.

\begin{lem}\label{lem:nonneighborhood}
For any graph $A$ and any vertex subset $U\subseteq V(A)$, if $|U| > \alpha_2(A)$ then $|\overline N(U)| \le 2 \alpha_2(A) + 1$.
\end{lem}
\begin{proof}
If not, let $U$ be a vertex subset for which $|U| \ge \alpha_2(A) + 1$ and $|\overline N(U)| \ge 2 \alpha_2(A) + 2$, and let $U'$ be some subset of $U$ of order $\alpha_2(A) + 1$. Since $|\overline N(U) \setminus U'| \ge \alpha_2(A) + 1$, there must exist a set $W\subseteq \overline N(U)$ disjoint from $U'$ of order $\alpha_2(A) + 1$ which has no edges to $U'$. Thus, $(U',W)$ is a bipartite independent set of order $\alpha_2(A)+1$, contradicting the definition of $\alpha_2(A)$.
\end{proof}

Henceforth, fix $m = n^{1/6}$, and $A$ will refer exclusively to the graph constructed in
Lemma~\ref{lem:A-construction} of this order $m$. In fact, we will only need the following three properties of $A$, which follow from Lemmas~\ref{lem:A-construction} through~\ref{lem:nonneighborhood}.

\begin{enumerate}
    \item The odd girth of $A$ is greater than $g$.
    \item If $U\subseteq V(A)$ and $|U| \ge m^{1-\frac{1}{3g}}$, then $|\overline N(U)| < m^{1-\frac{1}{3g}}$.
    \item If $U\subseteq V(A^{ 2})$ and $|U| \ge m^{2-\frac{1}{3g}}$, then $|\overline N(U)| < m^{2-\frac{1}{3g}}$.
\end{enumerate}

\subsection{Triangle-free colorings\label{subsec:triangle-free-colorings}}
Let $K_{n,n,n}^*$ be the complete tripartite directed graph (digraph), which has an edge in both directions between each pair of vertices in distinct parts. A {\it triangle-free coloring} of $K_{n,n,n}^*$ is a map $\chi: E(K_{n,n,n}^*)\rightarrow V(A)$ with no adjacent triangles. Let $\F$ be the family of all triangle-free colorings of $K_{n,n,n}^*$. The total number of colorings is $m^{6n^{2}}$, and our main lemma
is that the number of triangle-free colorings is much smaller.
\begin{lem}
\label{lem:main} There exists an absolute constant $a>0$ such that if $m=n^{\frac{1}{6}}$ and $n$ is sufficiently large in terms of $g$, then $|\F|\le m^{\left(6-\frac{a}{g}\right)n^{2}}$.
\end{lem}

The proof of Lemma~\ref{lem:main} is motivated by entropy counting methods (see the excellent survey of Galvin~\cite{Ga}), in particular Shearer's entropy inequality which is our Lemma~\ref{lem:Shearer} on page~\pageref{lem:Shearer} below. Shearer's inequality states that the cardinality of a finite set $S\subseteq \mathbb{Z}^n$ is bounded above by an appropriate product of the cardinalities of its projections to subspaces of $\mathbb{Z}^n$. Roughly speaking, we treat the family $\F$ of triangle-free colorings as a set of points in $\mathbb{Z}^{6n^2}$. We will show that $\F$ is small by breaking $\F$ into several subfamilies, each of which has small projection to some subspace of $\mathbb{Z}^{6n^2}$.

We first prove that Lemma~\ref{lem:main} implies Theorem~\ref{thm:oddgirthlowerbound}. The rest of the section is then devoted to proving Lemma~\ref{lem:main}.

\vspace{3mm}

\noindent {\it Proof that Lemma~\ref{lem:main} implies Theorem~\ref{thm:oddgirthlowerbound}.} 
Let $a$ be the constant in Lemma~\ref{lem:main}, and take $N = m^{\frac{a}{3g}n}=n^{\frac{a}{18g}n}$.

Recall
the $3$-graph $\Gamma(\chi)$ constructed in Section
\ref{sec:setup}, whose vertex set is $[N]$. We picked a uniform random function $\chi:[N]^{2}\rightarrow V(A)$,
and a triple $(i,j,k)$ appears as an edge of $\Gamma(\chi)$ when it forms an adjacent triangle of $\chi$.

Let $\ell$ be any odd integer between $3$ and $g$ inclusive. Since $A$ has odd girth greater than $g$, we know that $A$ is $\hom(C_\ell)$-free, so Lemma~\ref{lem:surj-free} shows that $\Gamma(\chi)$ is $L_{C_{\ell}}$-free. Thus, the link of each vertex in $\Gamma(\chi)$ has odd girth greater than $g$.

It remains to show that with positive probability, $\alpha_3(\Gamma(\chi)) < n$. We bound the probability that a fixed copy of $K_{n,n,n}^{(3)}$
induced by three disjoint sets $I,J,K\subseteq[N]$ appears in $\overline{\Gamma(\chi)}$. The event $X_{I,J,K}$ that none of the edges $(i,j,k)\in I\times J\times K$ appears
in $\Gamma(\chi)$ is exactly the event that the coloring $\chi$
contains no adjacent triangles $(i,j,k)\in I\times J \times K$. By Lemma~\ref{lem:main}, $\Pr[X_{I,J,K}] \le m^{-\frac{a}{g}n^2}$ for an absolute constant $a>0$.

There are fewer than $N^{3n}$ choices of disjoint sets $I, J, K \in {[N]\choose n}$. Taking a union bound over all these choices,
\[
\Pr[\overline{\Gamma(\chi)}\text{ contains a copy of }K_{n,n,n}^{(3)}]<N^{3n}m^{-\frac{a}{g}n^{2}} = 1.
\]
With positive
probability, $\overline{\Gamma(\chi)}$ contains no copy of $K_{n,n,n}^{(3)}$. In particular, there exists a $\chi$ for which $\alpha_3(\Gamma(\chi)) < n$.
This proves the theorem. \qed

\subsection{Low-Entropy Families \label{subsec:low-entropy-families}}

In this section we prepare for the proof of Lemma~\ref{lem:main} and sketch its main ideas. 

Let $E=E(K_{n,n,n}^{*})$. Given $\chi\in\F$ and a set of
edges $S\subseteq E$, define the family of recolorings of $\chi$ on $S$,
denoted $\chi^{*}(S)$, to be the set of all $\chi'\in\F$
with the property that $\chi|_{E\backslash S}=\chi'|_{E\backslash S}$.
In other words, $\chi^{*}(S)$ is the set of different ways to
recolor $\chi$ on the edges in $S$ while preserving the triangle-free property
of $\chi$.

Partition the set $E$ into three subsets $E_{1}$, $E_{2}$, $E_{3}$
of size $2n^{2}$, where $E_{1}=I\times(J\cup K)$ is the set of edges
out of $I$, $E_{2}=(J\cup K)\times I$ is the set of edges into $I$,
and $E_{3}=(J\times K)\cup(K\times J)$ is the set of edges between
$J$ and $K$. We define three families $\F_{1},\F_{2},\F_{3}\subseteq\F$. The set $\F_{i}$ will be the set of colorings
$\chi$ which are ``low-entropy'' on $E_{i}$. Loosely speaking, this means that if we think of $\F$ as a finite set in $\mathbb{Z}^{6n^2}$, $\F_i$ is chosen so that its projection onto the subspace indexed by $E_i$ is small.

Specifically, let $c=10^{-4}$, 
\begin{align*}
\F_{1} & = \{\chi\in\F:|\chi^{*}(E_{1})|\le m^{\left(2-\frac{c^2}{7g}\right)n^{2}}\},~\textrm{and}\\
\F_{3} & = \{\chi\in\F:|\chi^{*}(E_{3})|\le m^{\left(2-\frac{c^2}{7g}\right)n^{2}}\},
\end{align*}
so $\F_{1},\F_{3}$ are low-entropy in that they have few recolorings on the corresponding edge sets.

We define $\F_{2}$ using a different notion
of low-entropy. Let $\F_{2}\subseteq\F$
be the family of $\chi\in\F$ such that for each of $cn^{2}$
pairs $(j,k)\in J\times K$, there is a set $U_{j,k}\subseteq V(A)^{2}$
of size less than $m^{2-\frac{1}{3g}}$ such that
\[
|\{i\in I:(\chi(j,i),\chi(k,i))\in U_{j,k}\}|\ge cn.
\]
The family $\F_2$ is low-entropy in the sense that there are many triples $(i,j,k)$ for which the values of $(\chi(j,i), \chi(k,i))$ lie in a small set.

Our first lemma is that all three low-entropy families are small.
\begin{lem}
\label{lem:family-counting}For $m$ sufficiently
large,
\begin{equation}\label{eq:F-bound}
|\F_{i}| \le m^{\left(6-\frac{c^2}{7g}\right)n^2} 
\end{equation}
for each $i=1,2,3$.
\end{lem}

\begin{proof}
To show (\ref{eq:F-bound}) for $\F_1$, note that there are at most $m^{4n^{2}}$
ways to choose the values of a coloring on $E_{2}\cup E_{3}$, and
at most $m^{\left(2-\frac{c^2}{7g}\right)n^{2}}$ ways to extend such a coloring to $E_{1}$
to a coloring in $\F_{1}$. Thus,
\[
|\F_1|\le m^{\left(6-\frac{c^2}{7g}\right)n^2}.
\]
The same argument with $E_1$ and $E_3$ swapped proves (\ref{eq:F-bound}) for $\F_3$.

It remains to prove the lemma for $\F_2$. Among the $cn^{2}$ pairs
$(j,k)\in E_{2}$ for which $U_{j,k}$ exists, there must be a matching
$\{(j_{1},k_{1}),\ldots,(j_{cn/2},k_{cn/2})\}$ of size $cn/2$, so
that the values of $j_{i}$ are all distinct and the values of $k_{i}$ are all distinct. Therefore, for any $\chi\in \F_2$, there exists a $(cn/2)$-matching $M = \{(j_{1},k_{1}),\ldots,(j_{cn/2},k_{cn/2})\}$, a family $\mathcal U$ of sets $U_1,\ldots, U_{cn/2}\subset V(A)^2$ each of size at most $m^{2-\frac{1}{3g}}$, and a family $\mathcal I$ of sets $I_1,\ldots, I_{cn/2}\subset I$, each of size $cn$, for which $(\chi(j_t, i),\chi(k_t, i))\in U_t$ for all $t=1,\ldots,cn/2$ and all $i\in I_t$. We let $\F_2(M, \mathcal U, \mathcal I)$ be the family of all such $\chi$ for a particular choice of $(M, \mathcal U, \mathcal I)$, so that
\[
\F_2 \subseteq \bigcup_{(M,\mathcal U, \mathcal I)} \F_2(M,\mathcal U, \mathcal I).
\]

The number of choices of the matching $M$ is at most $(n^{cn/2})^2=n^{cn}$. The number of choices of each $U_t$ is at most $2^{m^2}$. The number of choices of each $I_t$ is at most ${n \choose cn} \le 2^n$. Altogether, there are at most
\[
n^{cn} \cdot (2^{m^2})^{cn/2} \cdot (2^{n})^{cn/2} = m^{o(n^2)}
\]
choices of $(M, \mathcal U, \mathcal I)$.

Once the above choices are made, we just need to bound the size of $\F_2(M,\mathcal U, \mathcal I)$. The number of ways to color all the edges $\chi(j_t,i)$ and $\chi(k_t,i)$
for a particular pair $(j_t,k_t)\in M$ is at most
\[
\left(m^{2-\frac{1}{3g}}\right)^{cn}\cdot \left(m^2\right)^{n-cn}=m^{\left(2-\frac{c}{3g}\right)n},
\]
since we must use one of the choices in $U_t$ when $i\in I_t$ and there are at most $m^2$ choices otherwise. Multiplying over the $cn/2$ values of $t$, this makes for a total of at most 
\[
\left(m^{\left(2-\frac{c}{3g}\right)n}\right)^{cn/2} = m^{\left(1-\frac{c}{6g}\right)cn^2}
\]
ways to color all the edges $\chi(j_t,i)$ and $\chi(k_t,i)$ with $1 \le t \le cn/2$ and $i\in I$.

So far we have colored a total of $cn^2$ edges of $K_{n,n,n}^*$. Each of the remaining $6n^2 - cn^2$ edges has at most $m$ possible colors to choose from. Thus,
\[
|\F_2(M, \mathcal U, \mathcal I)| \le m^{\left(1-\frac{c}{6g}\right)cn^2} \cdot m^{6n^2-cn^2} = m^{\left(6-\frac{c^2}{6g}\right)n^2}.
\]

But the number of choices of $(M,\mathcal U, \mathcal I)$ is $m^{o(n^2)}$, so for $n$ sufficiently large in terms of $g$,
\[
|\F_2 | \le m^{o(n^2)} \max_{(M,\mathcal U, \mathcal I)} |\F_2(M, \mathcal U, \mathcal I)| \le m^{\left(6-\frac{c^2}{7g}\right)n^2},
\]
as desired.
\end{proof}
We will soon construct an invariant $\iota:\F\rightarrow[0,1]$
with the property that for $m$ large enough,
\begin{align}
\mathbb{E}_{\F\backslash\F_{1}}[\iota(\chi)] & \ge \frac{3}{4} \label{eq:iota-F1} \\
\mathbb{E}_{\F\backslash(\F_{2}\cup\F_{3})}[\iota(\chi)] & \le \frac{1}{4}, \label{eq:iota-F23}
\end{align}
where $\mathbb{E}_{\mathcal{G}}[\cdot]$ is the averaging operator over a uniform random element of a family $\mathcal{G}$. To see why such an invariant is useful, consider its average value on the family $\F' = \F \setminus (\F_{1}\cup\F_{2}\cup\F_{3})$. If $\mathbb E_{\F'}[\iota(\chi)] \le \frac{1}{2}$, then (\ref{eq:iota-F1}) shows that $\F'$ contains at most half the complement of $\F_1$, so $|\F'| \le |\F_{2}\cup\F_{3}|$. Similarly, if $\mathbb E_{\F'}[\iota(\chi)] > \frac{1}{2}$, then (\ref{eq:iota-F23}) implies that $|\F'| \le |\F_1|$. In either case, at least half of $\F$ must lie in the union $\F_{1}\cup\F_{2}\cup\F_{3}$ of the low-entropy families. For the details of this argument, see the proof of Lemma~\ref{lem:main} on page~\pageref{proof:lem-main} below.

The next two lemmas give product formulas for the number of recolorings $\chi$ has on $E_1$ and $E_3$. The latter is easier to understand, so we begin with it.
\begin{lem}
\label{lem:F3}For any $\chi\in\F$, 
\[
|\chi^{*}(E_{3})|=\prod_{j\in J}\prod_{k\in K}|\chi^{*}(\{(j,k),(k,j)\})|.
\]
\end{lem}

\begin{proof}
This is just the observation that if we only change the values of
$\chi$ on edges between $J$ and $K$, whether or not any particular triple $(i,j,k)$ forms an adjacent
triangle depends only on the choice of $\chi(j,k)$ and $\chi(k,j)$, and not on
any other choices. Thus, we can pick a valid recoloring of those two pairs independently for each $(j,k)$.
\end{proof}
To prove the same kind of result for $E_{1}$, we associate a family
of bipartite graphs $(G_{i}(\chi))_{i\in I}$ to a given $\chi\in\F$.
The bipartite graph $G_{i}(\chi)$ has parts $J$ and $K$,
where $j\sim k$ in $G_{i}(\chi)$ if and only if $\chi(j,i)\sim\chi(j,k)$ and
$\chi(k,i)\sim\chi(k,j)$. That is, $j\sim k$ in $G_{i}(\chi)$ if
$(i,j,k)$ is two-thirds of the way to making an adjacent triangle,
only missing the adjacency $\chi(i,j)\sim\chi(i,k)$ centered at $i$. Observe that $G_i(\chi)$ depends only on $\chi|_{E_2 \cup E_3}$.

We claim that to count the number of ways to recolor $\chi$ on $E_1$, it suffices to enumerate homomorphisms out of $G_i(\chi)$. Write $\hom(G,A)$ for the set of all graph homomorphisms from $G$ to $A$.

\begin{lem}
\label{lem:F1}For any $\chi\in\F$,
\begin{equation}
|\chi^{*}(E_{1})|=\prod_{i\in I}|\hom(G_{i}(\chi),\overline{A})|.\label{eq:E1-prod}
\end{equation}
\end{lem}

\begin{proof}
The left hand side counts the number of ways to simultaneously recolor all
the edges $(i,j)$ and $(i,k)$ pointed out of $I$, while keeping the rest of $\chi$ fixed. Let $E_1(i)= \{i \}\times\{J\cup K\}$ be the set of edges pointed out of $i$. Since the edges out of $i$ can be recolored independently
of the edges out of $i'$, for any $i\ne i'$, we have the product formula
\begin{equation}\label{eq:E1i-prod}
|\chi^*(E_1)|=\prod_{i\in I}|\chi^*(E_1(i))|.
\end{equation}

It remains to show $|\chi^*(E_1(i))| = |\hom(G_{i}(\chi),\overline{A})|$. Since $G_{i}(\chi)$ keeps track
of which triples already have two adjacent pairs, a recoloring $\chi'$ of $\chi$ on $E_1(i)$ will be triangle-free if and only if $\chi'(i,j)\not\sim\chi'(i,k)$ whenever $j\sim k$
in $G_{i}(\chi)$. It follows that the map $v\mapsto \chi'(i,v)$ defined from $J\cup K$ to $V(A)$ must send edges of $G_{i}(\chi)$
to non-edges of $A$. This is exactly the definition of a graph homomorphism
from $G_{i}(\chi)$ to $\overline{A}$, so we have exhibited an injection $\chi^*(E_1(i))\hookrightarrow \hom(G_{i}(\chi),\overline{A})$.

Conversely, given any $\phi \in \hom(G_i(\chi),\overline{A})$, recolor $\chi$ on $E_1(i)$ to give a coloring $\chi'$ with $\chi'(i,v)=\phi (v)$. This $\chi'$ will always be triangle-free, so we have found a bijection between $\hom(G_i(\chi),\overline{A})$ and $\chi^*(E_1(i))$. Together with (\ref{eq:E1i-prod}), this completes the proof of (\ref{eq:E1-prod}).
\end{proof}
We are now ready to define the invariant $\iota$ mentioned at the beginning of this section, which should ``separate'' $\F\backslash \F_1$ from $\F \backslash (\F_2 \cup \F_3)$ in the sense of inequalities (\ref{eq:iota-F1}) and (\ref{eq:iota-F23}). The value of $\iota$ is a
probability computed over a uniform random choice of $i\in I$:
\[
\iota(\chi)\coloneqq\Pr_{i\in I}\left[|\hom(G_{i}(\chi),\overline{A})|\ge m^{\left(2-\frac{c^2}{g}\right)n}\right].\]
Recall that $\F_1$ is the family of colorings $\chi$ for which $|\chi^*(E_1)|$ is small, so by the preceding lemma this function $\iota(\chi)$ must be large for elements of $\F \backslash \F_1$. The next lemma formalizes this intuition by proving inequality (\ref{eq:iota-F1}).
\begin{lem}
\label{lem:iota-F1}If $\chi\in\F\backslash\F_{1}$,
then
\[
\iota(\chi)>\frac{3}{4}.
\]
\end{lem}

\begin{proof}
If not, there are at least $\frac{1}{4}n$ values of $i\in I$
for which $|\hom(G_{i}(\chi),\overline{A})|<m^{\left(2-\frac{c^2}{g}\right)n}$. For the other $i$, there is the trivial bound $|\hom(G_{i}(\chi),\overline{A})|\le m^{2n}$. By (\ref{eq:E1-prod}), we get
\begin{align*}
|\chi^{*}(E_{1})| & = \prod_{i\in I}|\hom(G_{i}(\chi),\overline{A})|\\
 & < \left(m^{2n}\right)^{3n/4}\cdot\left(m^{\left(2-\frac{c^2}{g}\right)n}\right)^{n/4}\\
 & = m^{\left(2-\frac{c^2}{4g}\right)n^2} \\
 & < m^{\left(2-\frac{c^2}{7g}\right)n^2},
\end{align*}
contradicting the fact that $\chi\not\in\F_{1}$.
\end{proof}

It remains to prove (\ref{eq:iota-F23}), which states that $\iota$ is typically small on $\F_2\cup \F_3$.

\subsection{Random Recoloring\label{subsec:Random-Recoloring}}

We come to one of the key ideas in this proof, which is that if we
randomly sample $\chi\in\F\backslash(\F_{2}\cup\F_{3})$,
then $G_{i}(\chi)$ behaves like a random subgraph of a dense bipartite graph, and each edge appears with probability at least $m^{-2}$. Such a random graph has with high probability very
few homomorphisms to $\overline{A}$, so $\iota(\chi )$
is usually small for these $\chi$.

We next define another family of bipartite graphs $(G_{i}^{*}(\chi))_{i\in I}$ with parts $J$ and $K$.
We have $j\sim k$ in $G_i^* (\chi)$ whenever there exists $\chi'\in\chi^{*}(\{(j,k),(k,j)\})$
for which $j\sim k$ in $G_{i}(\chi')$. Since $\chi \in \chi^{*}(\{(j,k),(k,j)\})$ itself, $G_{i}(\chi)$
is a subgraph of $G_{i}^{*}(\chi)$. In fact, we can say something
much stronger.
\begin{lem}\label{lem:random-subgraph}
If $\chi\in\F$ and $\chi'$ is a uniform random element
of $\chi^{*}(E_{3})$, then $G_{i}(\chi')$ is a random subgraph of
$G_{i}^{*}(\chi)$ where each edge of $G_{i}^{*}(\chi)$ appears in
$G_{i}(\chi')$ independently with probability at least $m^{-2}$.
\end{lem}

\begin{proof}
Every edge of $G_{i}(\chi')$ is an edge of $G_{i}^{*}(\chi)$
by the definition of $G_{i}^{*}$. On the other hand, a given edge $\{j, k\}$
appears in $G_{i}^*(\chi)$ if and only if there exists an element $\chi'\in\chi^{*}(\{(j,k),(k,j)\})$
for which $\chi'(j,i)\sim\chi'(j,k)$ and $\chi'(k,i)\sim\chi'(k,j)$.
Thus the appearance of $j\sim k$ in $G_{i}(\chi')$ depends only
on $\chi'(j,k)$ and $\chi'(k,j)$, whose values are independent
from those random choices associated to all the other ordered pairs in $E_3$.

The total number of choices for $(\chi'(j,k),\chi'(k,j))$ is at most $m^2$, and the choice is made uniformly among all possible ones. Since at least one choice makes $j\sim k$ in $G_i(\chi')$, the probability that this happens is at least $m^{-2}$.
\end{proof}

We next show that $G_{i}^{*}(\chi)$ is usually dense when $\chi\not \in \F_{2}\cup\F_{3}$.

\begin{lem}
\label{lem:Gi-dense}If $\chi\in\F\backslash(\F_{2}\cup\F_{3})$
and $n$ is sufficiently large, then for at least $(1-3c)n^{3}$
triples $(i,j,k)\in I\times J\times K$, we have $j\sim k$ in $G_{i}^{*}(\chi)$.
\end{lem}

\begin{proof}
For each $(j,k)\in J\times K$, define $U_{j,k}$ to be the set of all $(u_{1},u_{2})\in V(A^{2})$
non-adjacent in $A^2$ to all pairs of the form $(\chi'(j,k), \chi'(k,j))$ for $\chi' \in \chi^{*}(\{(j,k),(k,j)\})$. By Property 3 of $A$ defined at the end of Section~\ref{subsec:A-construction}, either $|U_{j,k}| < m^{2-\frac{1}{3g}}$ or $|\chi^{*}(\{(j,k),(k,j)\})| < m^{2-\frac{1}{3g}}$.

Also, let $I_{j,k}$ be the set of $i\in I$ for which $j\not \sim k$ in $G_i^*(\chi)$. By the definition of $G_i^*(\chi)$, $i\in I_{j,k}$ is equivalent to $(\chi(j,i),\chi(k,i))\in U_{j,k}$.

We bound the number of triples $(i,j,k)$ for which $j\not\sim k$ in $G_i^*(\chi)$ by breaking them up into three types. Type 1 triples are those for which $|\chi^{*}(\{(j,k),(k,j)\})| < m^{2-\frac{1}{3g}}$. Type 2 triples are those for which $|U_{j,k}| < m^{2-\frac{1}{3g}}$ and $|I_{j,k}|\ge cn$. Type 3 triples are the remaining ones, which satisfy $|I_{j,k}|< cn$.

Using Lemma~\ref{lem:F3} and the fact that $\chi\not\in\F_{3}$,
we get
\[
\prod_{j\in J}\prod_{k\in K}|\chi^{*}(\{(j,k),(k,j)\})| = |\chi^*(E_3)| > m^{(2-\frac{c}{7g})n^2} > m^{(2-\frac{c}{3g})n^2}.
\]
It follows that there are less than $cn^{2}$ pairs $(j,k)\in J\times K$ for which $|\chi^{*}(\{(j,k),(k,j)\})|<m^{2-\frac{1}{3g}}$, as otherwise the above product would be less than $\left(m^{2-\frac{1}{3g}}\right)^{cn^2}\left(m^2\right)^{n^2-cn^2}=m^{(2-\frac{c}{3g})n^2}$, a contradiction. Thus there are at most $cn^3$ Type 1 triples.

As for Type 2 triples, note that if there are at least $cn^2$ pairs $(j,k)$ for which $|U_{j,k}| < m^{2-\frac{1}{3g}}$ and $|I_{j,k}|\ge cn$, then this contradicts the assumption that $\chi \not \in \F_2(c)$. Thus there are fewer than $cn^2$ choices of $j$ and $k$ for Type $2$ triples, for a total of at most $cn^3$.

Finally, Type 3 triples satisfy $|I_{j,k}|< cn$, so there are at most $cn$ choices of $i$ for each pair $(j,k)$. The number of Type 3 triples is also at most $cn^3$.

In total, there are at most $3cn^3$ triples $(i,j,k)$ for which $j\not\sim k$ in $G_i^*(i)$.
\end{proof}

Define a bipartite graph $G=(U,V,E)$ to be $m$-good if
every vertex of $U$ has degree $m$ and for every $V'\subseteq V$ of
size $|V'|\ge (1-c)|V|$, at least $\frac{1}{4}|E|$ edges
are incident to $V'$. The next lemma is the only place in the proof where we need to pick $c$ fairly small; everywhere else $c=\frac{1}{36}$ would suffice.

When $\chi'$ is a uniform random element of $\chi^*(E_3)$ and $i$ is a uniform random element of $I$, by the previous two lemmas we know that $G_i(\chi')$ is distributed like a random subgraph of the dense bipartite graph $G_i^*(\chi)$. Using this, we show that with high probability $\chi'$ has few homomorphisms to $\overline{A}$, so that $\iota(\chi')$ is usually small. 

\begin{lem}
\label{lem:m-good-sub}Suppose $\chi\in\F$, $G_{i}^{*}(\chi)$
has at least $n^{2}/3$ edges, and $\chi'$ is a uniform random element
of $\chi^{*}(E_{3})$. Then, with high probability $G_{i}(\chi')$ contains an $m$-good
subgraph whose vertex set contains at least $n/6$ vertices of $J$ and all the
vertices of $K$.
\end{lem}

\begin{proof}
By Lemma~\ref{lem:random-subgraph}, $G_i(\chi')$ is a random subgraph of $G_i^*(\chi)$ where edges appear independently, and the probability any given edge appears is at least $m^{-2}$. Define $L^*$ to be the random subgraph of $G_i^*(\chi)$ where each edge appears with probability exactly $m^{-2}$. Since every subgraph is at least as likely to appear in $G_i(\chi')$ as in $L^*$, it suffices to show that with high probability $L^*$ contains an $m$-good subgraph with the given properties. 

Since $G_{i}^{*}(\chi)$ has at least $n^{2}/3$ edges, there is a subset $J' \subset J$ of $n/6$ vertices of degree at least $n/6$. For $n$ sufficiently
large, since $m=n^{1/6}$ and each edge of $G_{i}^{*}(\chi)$ appears
in $L^*$ with probability $m^{-2}=n^{-1/3}$, by the union bound, with high probability every vertex of $J'$ has degree at least $m$ in $L^*$. We may condition on this event occurring.

Let $L$ be the graph obtained from $L^*[J'\cup K]$ by independently
and uniformly selecting exactly $m$ edges incident to each $j\in J'$. By the definition of $L^*$, the neighborhood of any given $j\in J'$ in $L^*$ is a uniform random subset of size at least $m$ of the neighborhood of $j$ in $G_i^*(\chi)$. Thus, the neighborhood of any given $j\in J'$ in the graph $L$ is exactly a uniform random $m$-subset of its neighborhood in $G_i^*(\chi)$. We claim that $L$ is $m$-good with high probability.

For each $K'\subseteq K$ of order at least $\frac{11}{12}n$, write
$e_L(J',K')$ for the number of edges in $L$ incident to $K'$. We
bound the probability that $e_L(J',K')\le\frac{1}{4}|E(L)|=mn/24$
edges. First note that since each vertex of $J'$ has degree at least $n/6$ in $G_i^*(\chi)$, at least half of its neighbors are in $K'$. We claim that the number of edges $K'$ receives in $G$ will be tightly concentrated about its mean, which is
at least $mn/12$. In fact, if $X$ is the number of edges incident
to $K'$, then 
\[
X= \sum_{j\in J'} X_j,
\]
where $X_{j}$ is the random variable counting the
edges of $L$ between $j$ and $K'$. Each $X_{j}$ takes values
in $[0, m]$, and since at least half of the neighbors of $j$ are in $K'$, $\mathbb{E}[X_j] \ge m/2$. Thus, $\mathbb{E}[X] \ge (m/2)(n/6)=mn/12$. We wish to show that $X$ is tightly concentrated about this mean.

One standard form of the Chernoff bound (see Lemma A.1.16 of Alon and Spencer~\cite{AlSp}) states if $Y$ is the sum of $t$ mutually independent random variables $Y_1,\ldots, Y_t$ satisfying $\mathbb E[Y_j]=0$ and $|Y_j|\le 1$, then
\[
\Pr[Y < -a] < e^{-a^2/2t}
\]
for all $a\ge 0$.

Taking $t = n/6$, $a=n/24$, and $Y_j = \frac1m (X_j - \mathbb E[X_j])$, we find
\[
\Pr[e_L(J',K') < mn/24] \le \Pr[Y < -a] < e^{-a^2/2t} = e^{-n/192} \le 1.005^{-n}.
\]

On the other hand, it is easy to check that for $c=10^{-4}$, the number of subsets $K'\subseteq K$ of order
$(1-c)n$ is at most
\[
\binom{n}{(1-c)n}\le1.002^{n}.
\]
By the union bound over all $K'$ of this size, we see that with high probability, $e_L(J',K')\ge mn/24$
for every such $K'$, proving that $L$ is $m$-good.
\end{proof}
Finally, we need to prove that $m$-good graphs have few homomorphisms
to $\overline{A}$. To this end, we will use the following corollary of an entropy lemma of Shearer~\cite{ChFrGrSh}
(see also Corollary 15.7.5 of Alon and Spencer~\cite{AlSp}).
\begin{lem}
\label{lem:Shearer}Let $S_1, \ldots, S_n$ be finite sets, and let $\F$ be a family of $n$-tuples in $S_{1}\times S_{2}\times\cdots\times S_{n}.$
Let $\mathcal{C}=\{C_{1},\ldots,C_{r}\}$ be a collection of subsets
of $\{1,\ldots,n\}$ and suppose that each $1\le i\le n$ belongs
to at least $k$ members of $\mathcal{C}$. For each $1\le j\le r$
let $\F_{j}$ be the set of all projections of $\F$
onto the coordinates in $C_{j}$. Then,
\[
|\F|^{k}\le\prod_{j=1}^{r}|\F_{j}|.
\]
\end{lem}

We are ready to show the last step.
\begin{lem}
\label{lem:good-low-hom}If $m$ is sufficiently large in terms of $g$ and $L=(U,V,E)$ is an $m$-good graph with
$|U|= n/6$ and $|V|=n$, then
\[
|\hom(L,\overline{A})|<m^{\left(\frac{7}{6}-\frac{c}{4g}\right)n}.
\]
\end{lem}

\begin{proof}
If $\phi\in\hom(L,\overline{A})$ and $S\subseteq U\cup V$, define $\phi^{*}(S)$ to
be the set of all homomorphisms $\phi'$ which agree with $\phi$
outside $S$. Then, we see
by the trivial bound $m^{n/6}$ on the number of mappings from $U$ to the vertex set of $\overline{A}$ that
\begin{equation}\label{eq:phi-negligible}
|\{\phi:|\phi^{*}(V)|<m^{\left(1-\frac{c}{3g}\right)n}\}| <  m^{n/6}m^{\left(1-\frac{c}{3g}\right)n} < \frac{1}{2}m^{\left(\frac{7}{6}-\frac{c}{4g}\right)n}
\end{equation}
when $m$ is sufficiently large. It therefore suffices to count $\phi$ for which $|\phi^{*}(V)|\ge m^{(1-\frac{c}{3g})n}$. For such a $\phi$, we claim that the set $V'$ of vertices $v \in V$ for which $|\phi^{*}(\{v\})|\ge m^{1-\frac{1}{3g}}$ satisfies $|V'| \geq (1-c)n$. Indeed, if otherwise, as $V$ is an independent set, then
\[
|\phi^*(V)| = \prod_{v\in V} |\phi^*(\{v\})| = \prod_{v\in V'} |\phi^*(\{v\})| \cdot \prod_{v\in V\setminus V'} |\phi^*(\{v\})| \le m^{|V'|n+|V\setminus V'|\left(1-\frac{1}{3g}\right)n} < m^{\left(1-\frac{c}{3g}\right)n},
\]
which is a contradiction.

Applying Property 2 of $A$ defined at the end of Section~\ref{subsec:A-construction}, for each $v\in V'$, the size of $\phi(N(u))$ must be less than $m^{1-\frac{1}{3g}}$.
Also, since $L$ is $m$-good and $|V'|\ge (1-c)n$, $V'$ is incident
to at least $mn/24$ edges of $L$.

We condition on the choice of $V'$. Let $\F(V')$ be
the set of possible restrictions $\phi|_{U}$ of homomorphisms $\phi$ for
which $|\phi^{*}(\{v\})|\ge m^{1-\frac{1}{3g}}$ on $v\in V'$. Letting $C_{v}=N(v)$ for each $v\in V$, we will use Lemma~\ref{lem:Shearer}
to bound the size of $\F(V')$. We think of $\F(V')$ as a family of vectors in $V(A)^{U}$, and $\F_{v}(V')$ will be the projection of $\F(V')$ onto the coordinates in $C_v = N(v)$.

Since $L$ is $m$-good, each $u\in U$ appears in exactly $m$ of the sets $C_{v}$.
Thus,
\[
|\F(V')|^{m}\le\prod_{v\in V}|\F_{v}(V')|.
\]

We also know that if $v\in V'$, then
\[
|\F_{v}(V')|<2^{m}\left(m^{1-\frac{1}{3g}}\right)^{\deg v},
\]
since there are at most $2^m$ choices of the set of colors $\phi(u)$ to appear in $N(v)$, and at most $m^{1-\frac{1}{3g}}$ colors to use on each vertex. Using the trivial bound for vertices outside $V'$, we get
\[
|\F(V')|^{m} \le\prod_{v\in V}|\F_{v}(V')| < 2^{mn}\prod_{v\in V'}m^{\left(1-\frac{1}{3g}\right)\deg v}\prod_{v\not\in V'}m^{\deg v}.
\]
Because the total degree from $V$ is just $|E|=mn/6$ and at least $mn/24$ of these edges are incident to $V'$, this inequality reduces to
\begin{align*}
|\F(V')|^m & \le 2^{mn}m^{\frac{1}{6}mn - \frac{1}{3g}\cdot \frac{1}{24}mn}, \\
|\F(V')| & \le 2^n m^{\left(\frac{1}{6} - \frac{1}{72g}\right)n}.
\end{align*}
There are at most $2^n$ choices of $V'$, and for each element $\phi_U\in \F(V')$ there are at most $m^n$ ways to extend it to $V$. Thus, for $m$
sufficiently large,
\begin{align*}
|\{\phi:|\phi^{*}(V)|\ge m^{\left(1-\frac{c}{3g}\right)n}\}| & \le m^{n}\sum_{V'}|\F(V')|\\
 & \le 2^{2n}m^{\left(\frac{7}{6}-\frac{1}{72g}\right)n}\\
 & \le \frac{1}{2}m^{\left(\frac{7}{6}-\frac{c}{4g}\right)n}.
\end{align*}
Thus, with (\ref{eq:phi-negligible}), we see that there are at most $m^{\left(\frac{7}{6}-\frac{c}{4g}\right)n}$ homomorphisms $\phi$ from $L$ to $\overline{A}$.
\end{proof}

We have all the ingredients to complete the proof of the main lemma.

\begin{proof}[Proof of Lemma~\ref{lem:main}.] \label{proof:lem-main}
Explicitly, we will show that if $m=n^{\frac16}$, $n$ is sufficiently large in terms of $G$, and $\F$ is the family of triangle-free colorings of $K_{n,n,n}^*$, then
\[
|\F| \le m^{\left(6-\frac{c^2}{8g}\right)n^2},
\]
where $c=10^{-4}$ as before.

With 
\[
\iota(\chi)=\Pr_{i\in I}\left[|\hom(G_{i}(\chi),\overline{A})|\ge m^{2-\frac{c^2}{g}}\right],
\]
Lemma~\ref{lem:iota-F1} shows
\begin{equation}\label{eq:iota-big}
\mathbb{E}_{\F\backslash\F_{1}}[\iota(\chi)]\ge\frac{3}{4}.
\end{equation}
On the other hand, if $\chi\in\F\backslash(\F_{2}\cup\F_{3})$
we know by Lemma~\ref{lem:Gi-dense} that there are a total of at least
$(1-3c)n^{3}\ge\frac{11}{12}n^{3}$ edges among the graphs $G_{i}^{*}(\chi)$.
In particular, for at least $\frac{7}{8}n$ values of $i$, $|E(G_{i}^{*}(\chi))|\ge n^{2}/3$.
By Lemma~\ref{lem:m-good-sub}, for such an $i$, if $\chi'$ is
randomly resampled from $\chi^{*}(E_{3})$ then with high probability $G_{i}(\chi')$
contains an $m$-good subgraph $L$ satisfying the conditions of Lemma
\ref{lem:good-low-hom}. Thus, for these $i$, with high probability
\[
|\hom(G_{i}(\chi'),\overline{A})|\le m^{\frac{5}{6}n}\cdot |\hom(L,\overline{A})|\le m^{\left(2-\frac{c}{4g}\right)n} < m^{\left(2-\frac{c^2}{g}\right)n}
\]
by Lemma~\ref{lem:good-low-hom}, and thus $\iota(\chi')=0$. For the remaining at most $\frac{1}{8}n$ values of $i$, we just bound $\iota(\chi')\le 1$. It follows that
\[
\mathbb{E}_{\chi' \in \chi^*(E_3)}[\iota(\chi')]\le\frac{1}{8}\cdot 1 + \frac{7}{8} \cdot o(1) = \frac{1}{8} + o(1),
\]
where the error term goes to zero as $m\rightarrow\infty$.

But $\chi'$ is just a uniform random element of $\chi^*(E_3)$, so $\chi'|_{E_1 \cup E_2} = \chi|_{E_1 \cup E_2}$. In particular, $\chi'(E_3)=\chi(E_3)$, and if $\chi \not\in \F_{2}(c)$ we also know $\chi \not\in \F_{2}(c)$. If originally $\chi$ was chosen out of $\F\backslash(\F_{2}\cup\F_{3})$ uniformly at random, then the marginal distribution of $\chi'$ is also uniformly random from $\F\backslash(\F_{2}\cup\F_{3})$,
so we have
\begin{equation}
\mathbb{E}_{\F\backslash(\F_{2}\cup\F_{3})}[\iota(\chi)]\le\frac{1}{8}+o(1) \le \frac{1}{4},\label{eq:iota-small}
\end{equation}
when $m$ is large enough.

Suppose $\F' = \F \setminus (\F_{1}\cup\F_{2}\cup\F_{3})$. If $\mathbb E_{\F'}[\iota(\chi)] \le \frac{1}{2}$, then we claim that $\F'$ contains at most half the complement of $\F_1$. If not,
\[
\E_{\F \setminus \F_1}[\iota(\chi)] < \frac{1}{2} \E_{\F'}[\iota(\chi)] + \frac{1}{2} \cdot 1 \leq \frac{3}{4},
\]
which contradicts~(\ref{eq:iota-big}). Thus, $|\F'| \le |\F_{2}\cup\F_{3}|$. If $\mathbb E_{\F'}[\iota(\chi)] > \frac{1}{2}$, then we can show by the same argument that (\ref{eq:iota-small}) implies $|\F'| \le |\F_1|$.

Either way, we see that
\[
|\F'|\le\max(|\F_{1}|,|\F_{2}\cup\F_{3}|)\le|\F_{1}\cup\F_{2}\cup\F_{3}|,
\]
whence by Lemma~\ref{lem:family-counting},
\[
|\F|\le2|\F_{1}\cup\F_{2}\cup\F_{3}|\le m^{\left(6-\frac{c^2}{8g}\right)n^{2}},
\]
for $m$ sufficiently large, completing the proof.
\end{proof}

\section{Proof of Theorem~\ref{thm:linkcliqueversusclique}} \label{sec:link-upper-bound}
In this section, we prove Theorem~\ref{thm:linkcliqueversusclique} using the vertex online Ramsey game defined by Conlon, Fox, and Sudakov~\cite{CoFoSu}. We generalize their definition of the game to all (vertex) ordered graphs.

For two ordered graphs $H_1, H_2$, the vertex online Ramsey game is a graph-building game played between two players Builder and Painter. The games starts from the empty graph, and at step $i$ a new vertex $v_i$ is revealed. For every existing vertex $v_j$, $j=1,\ldots, i-1$, Builder decides, in order, whether or not to draw the edge $\{v_j, v_i\}$. If he does draw the edge, Painter has to color it either red or blue immediately. Builder wins when the current graph contains either a red (ordered) copy of $H_1$ or a blue (ordered) copy of $H_2$.

Conlon, Fox and Sudakov studied the game in the case when $H_1$ and $H_2$ are complete graphs. In our application, we will pick $H_1$ to be the {\it forward star} $K_{1,s-1}^*$, which we define to be the ordered graph on $s$ vertices and $s-1$ edges where the first vertex is adjacent to each of the others. The other graph $H_2$ we will pick to be a complete graph $K_{n-1}$. Thus Builder wins when there is either a vertex $v_i$ with $s-1$ red edges to later vertices $v_j > v_i$, or a blue $K_{n-1}$.

The next lemma is a straightforward modification of Theorem 2.1 from \cite{CoFoSu}, which connects the vertex online Ramsey game to the Ramsey numbers of $3$-graphs.

\begin{lem}\label{lem:vertex-online}
Suppose in the vertex on-line Ramsey game that Builder has a strategy which ensures
a red $K_{1,s-1}^*$ or a blue $K_{n-1}$ using at most $v$ vertices, $r$ red edges, and in total $m$ edges. Then, for any $0 < \alpha \le \frac{1}{2}$,
\[
r(L_{K_s}, K_n^{(3)}) \leq (v+1)\alpha^{-r}(1-\alpha)^{r-m}.
\]
\end{lem}
\begin{proof}
Let $N=(v+1)\alpha^{-r}(1-\alpha)^{r-m}$, and let $\Gamma$ be a $3$-graph on $N$ vertices. We will use the given Builder strategy to find either a $L_{K_s}$ in $\Gamma$ or a $K_n^{(3)}$ in $\overline{\Gamma}$. 

To do so, we will pick out vertices $v_1,\ldots, v_h$ of $\Gamma$ one at a time in a certain well-defined vertices explained later. Builder uses his strategy to build an auxiliary graph $G$ on these $v_i$ as they arrive, and Painter will color the edges using a deterministic rule that we explain later. The two-coloring of $\Gamma$ will have the property that if $1\le i < j \le h$ and $\{v_i,v_j\} \in E(G)$ is red, then for all $j < k \le h$, $\{v_i,v_j,v_k\} \in E(\Gamma)$. Also, if $\{v_i,v_j\} \in E(G)$ is blue, then $\{v_i,v_j,v_k\} \not\in E(\Gamma)$ for all $j < k \le h$. 

After the $a$-th step of the process, we have picked $v_1,\ldots, v_a$, and we keep track of a set $S$ of candidates for vertex $v_{a+1}$. This set $S$ is defined to contain all $v\in V(\Gamma)\setminus \{v_1,\ldots, v_a\}$ such that for all $1\le i < j \le h$, $\{v_i,v_j,v\} \in E(\Gamma)$ if $\{v_i,v_j\} \in E(G)$ is red and $\{v_i,v_j,v\} \not\in E(\Gamma)$ if $\{v_i,v_j\} \in E(G)$ is blue. Thus every $w\in S$ is a valid choice for $v_{a+1}$, and in the beginning of step $a+1$, we pick any such $w$ to be $v_{a+1}$ and delete it from $S$. To start the process, in the first step we pick vertex $v_1$ arbitrarily and we have $S= V(\Gamma) \setminus \{v_1\}$. 

Then, when Builder draws a new edge between $v_i$ and $v_j$, Painter counts the number of triples of the form $\{v_i, v_j, v\}$ which are edges of $\Gamma$, where $v$ ranges through the available candidates $S$. If at least $\alpha |S|$ of these triples are edges, then Painter colors $\{v_i, v_j\}$ red. Otherwise, Painter colors it blue. In both cases, $S$ is replaced by the appropriate subset of valid candidates for $v_{a+1}$ after the new edge is drawn and colored.

The rules above imply that each time a red edge is drawn, $|S|$ shrinks by at most a factor of $\alpha$, and each time a blue edge is drawn, $|S|$ shrinks by at most a factor of $1-\alpha$. Also, a single vertex is removed from $S$ every time a new vertex $v_{a+1}$ is added to $G$.

Builder and Painter continue playing this game until either Builder wins or we run out of candidate vertices in $S$.

Builder's strategy ensures finding a red $K_{1,s-1}^*$ or a blue $K_{n-1}$ using at most $v$ vertices, $r$ red edges, and $m$ total edges. Recall that $\alpha \le \frac{1}{2}$ and initially we chose
\[
|S| = N = (v+1)\alpha^{-r}(1-\alpha)^{r-m}.
\]
The size of the candidate set $S$ is reduced to at least $\alpha (|S|-1)$ whenever an edge is colored red, and at least $(1-\alpha)(|S|-1)$ whenever an edge is colored blue. Also, $S$ decreases by one at the beginning of each step when we pick the vertex to add to the sequence. It follows from our choice of $N$ that Builder will be able to win before $S$ becomes empty. If $h$ vertices have been built at the end of the game, since $S$ is still nonempty we can pick an arbitrary vertex of $S$ and call it $v_{h+1}$.

Suppose at the end of the game that $G$ contains a red forward star $K_{1,s-1}^*$ on vertices $v_{i_1},\ldots, v_{i_s}$, where $v_{i_1}$ has a red edge to each of the remaining vertices and $i_1<i_2<\ldots<i_s \leq h$. Then, every triple $(v_{i_1}, v_j, v_k)$, where $j < k$ are elements of $\{i_2,\ldots, i_s, h+1\}$, must be an edge of $\Gamma$. Thus $\Gamma$ contains a copy of $L_{K_s}$ on the vertices $v_{i_1},\ldots, v_{i_s}, v_{h+1}$.

Otherwise, suppose $G$ contains a blue clique $K_{n-1}$ on vertices $v_{i_1},\ldots, v_{i_{n-1}}$ with $i_1<i_2<\ldots<i_{n-1}<h+1$. Then, no triple $(v_i, v_j, v_k)$, where $i<j<k$ are elements of $\{i_1,\ldots, i_s, h+1\}$, can be an edge of $\Gamma$. Thus $\overline{\Gamma}$ contains a copy of $K_n^{(3)}$ on the vertices $v_{i_1},\ldots, v_{i_{n-1}}, v_{h+1}$.
\end{proof}

To prove Theorem~\ref{thm:linkcliqueversusclique}, it suffices to show there is an appropriate Builder strategy to apply Lemma~\ref{lem:vertex-online}, which is given by the following lemma. 

\begin{lem}\label{lem:vertex-online1}
In the vertex on-line Ramsey game, Builder has a strategy which ensures a red $K_{1,s-1}^*$ or a blue $K_{n-1}$ using at most $n-1+(s-2)(n-2)$ vertices, $(s-2)(n-2)+1$ red edges, and in total $(s-1){n-1 \choose 2}$ edges. 
\end{lem}
\begin{proof}
Consider the following Builder strategy in the vertex online Ramsey game. Builder keeps track of two vertex subsets, $U$ and $W$, where $U$ is a blue clique of order at most $n-1$ and every vertex of $W$ has a red edge to some vertex in $U$. 

Initially, $U=\{v_1\}$ is a single vertex and $W$ is empty. Each time a new vertex $v_a$ is introduced, Builder draws all edges between $v_a$ and $U$, stopping if an edge is painted red. If Painter paints all of them blue, then $v_a$ is added to $U$. Otherwise, $v_a$ is added to $W$. In this way, $U$ and $W$ partition the set of vertices that are introduced. We let $G$ denote the current graph, which has vertex set $U \cup W$.  

We claim that Builder has already won if $|U|\ge n-1$ or $|W| > (s-2)(n-2)$. Indeed, if $|U| \ge n-1$, then $G[U]$ contains a blue $K_{n-1}$. Otherwise, if $|W| > (s-2)(n-2) \ge (s-2) |U|$, then some vertex in $U$ must have at least $s-1$ red edges to $W$. These edges form a red forward star $K_{1,s-1}^*$ because red edges out of a vertex of $U$ can only be drawn to vertices of $W$ that appeared later in the game.

Therefore, Builder's strategy ensures a red $K_{1,s-1}^*$ or a blue $K_{n-1}$ using at most $ n-1 + (s-2)(n-2)$ vertices. Since there are no red edges within $U$ and each vertex of $W$ is in at most one red edge, Builder uses at most $(s-2)(n-2)+1$ red edges. Before the last vertex is added, there are at most $n-2$ vertices in $U$, and the $i^{\textrm{th}}$ vertex in $U$ is in at most $s-2$ red edges, and each of these at most $s-2$ vertices are in exactly $i$ edges. So, before the last vertex, at most ${n-2 \choose 2}$ edges are in $U$, and at most $\sum_{i=1}^{n-2}i(s-2)={n-1 \choose 2}(s-2)$ edges contain a vertex in $W$. The last vertex added in the game is in at most $n-2$ edges. In total, the number of edges in the game is at most 
\[{n-2 \choose 2}+{n-1 \choose 2}(s-2)+n-2=(s-1){n-1 \choose 2}.\] 
\end{proof}

This is sufficient to prove the upper bound for Theorem~\ref{thm:linkcliqueversusclique}.

\begin{proof}[Proof of Theorem~\ref{thm:linkcliqueversusclique}.]

Apply Lemma~\ref{lem:vertex-online} with $\alpha=1/n$, $v=n-1+(s-2)(n-2)<sn$, $r=(s-2)(n-2)+1$, and $m=(s-1){n-1 \choose 2}$ to get
\[
r(L_{K_s}, K_n^{(3)}) < (sn)\alpha^{-r}(1-\alpha)^{-m}<
(n\sqrt{e})^{sn}<(2n)^{sn},\]
where in the second inequality we used  $(1-\frac{1}{n})^{n-1}>1/e$. 
\end{proof}

\section{Recursive bounds on hypergraph Ramsey numbers} \label{sec:recursive}

Let $H,G,F$ be $k$-graphs. For a vertex $v$ of $H$ and a positive integer $t$, let $H(v,t)$ be the $k$-graph on $|H|+t-1$ vertices formed by adding $t-1$ copies of $v$ to $H$. Let $H(v,F)$ be the $k$-graph on $|H|+|F|-1$ vertices formed by adding $|F|-1$ copies of $v$ to $H$ which together with $v$ induce a copy of $F$.

We have the following recursive bounds on Ramsey numbers of blow-up hypergraphs versus other graphs. Special cases of these results have already been observed in \cite{CoFoSu,ErHa,MuRo}. The proof is essentially the same, and included here for completeness. 

\begin{prop}
\label{prop:blowup}
If $H$ is a $k$-graph on $h$ vertices, $v$ is a vertex of $H$, $G$ is a $k$-graph, and $t$ is a positive integer, then $r(H(v,t),G) \leq t\cdot 2^h \cdot r(H,G)^{h-1}$. 
\end{prop}
\begin{proof}
Let $N = t \cdot 2^h \cdot r(H,G)^{h-1}$ and suppose $\Gamma$ is an $H(v,t)$-free $k$-graph on $N$ vertices. Let $\Gamma_H$ be the $h$-graph on $V(\Gamma)$ where an $h$-tuple $\{v_1,\ldots, v_h\}$ is an edge of $\Gamma_H$ if the induced subhypergraph $\Gamma[\{v_1,\ldots,v_h\}]$ contains $H$. Because $\Gamma$ is $H(v,t)$-free, it follows that for every copy of $H\setminus \{v\}$ in $\Gamma$, there are fewer than $t$ ways to extend it to a copy of $H$. As the total number of copies of $H\setminus \{v\}$ in $\Gamma$ is at most $N^{h-1}$, we see that $\Gamma_H$ has at most $tN^{h-1}$ edges.

Now, we can find an independent set of $\Gamma_H$ by picking a random vertex subset $S$, each vertex with probability $p$, and then removing at most one vertex from each edge in $S$. The expected number of edges in $S$ is at most $p^h t N^{h-1}$. We pick $p$ to satisfy $p^{h-1} = N^{2-h}/2t$, whence the expected number of vertices deleted is at most
\[
p^h tN^{h-1} = \frac{1}{2} p N,
\]
and therefore $\Gamma_H$ has an independent set $I$ of size at least
\[
\frac{1}{2}pN = \frac{1}{2}\Big(\frac{N}{2t}\Big)^{\frac{1}{h-1}} = r(H, G).
\]
By definition of $\Gamma_H$, $\Gamma[I]$ is an $H$-free $k$-graph of size $r(H,G)$, so $\overline{\Gamma[I]}$ contains a copy of $G$, as desired.
\end{proof}

The previous proposition allows us to replace a vertex $v$ of $H$ by an independent set. The next one generalizes this to replacing $v$ by any $k$-graph.

\begin{prop}\label{prop:replace}
For $k$-graphs $H$, $G$, and $F$ with $H$ having $h$ vertices, vertex $v$ of $H$, and positive integer $t$, we have $r(H(v,F),G) \leq r(F,G)\cdot 2^h \cdot r(H,G)^{h-1}$. 
\end{prop}
\begin{proof}
Let $t=r(F,G)$. Consider a red-blue edge-coloring of the complete $k$-graph on $N=t\cdot 2^h \cdot r(H,G)^{h-1}$ vertices. We will show that such a coloring must have a red $H(v,F)$ or a blue $G$. Apply Proposition \ref{prop:blowup}, so the red-blue edge-coloring has a red $H(v,t)$ or a blue $G$. In the latter case, we are done, so we may assume there is a red $H(v,t)$. Among the $t=r(F,G)$ vertices forming copies of $v$ in the red $H(v,t)$, there is a red $F$ or a blue $G$. If there is a red $F$, we get that this red $F$ together with the vertices forming the copy of $H \setminus \{v\}$ in the red $H(v,t)$ form a red $H(v,F)$. Otherwise, there is a blue $G$, in which case we are also done. 
\end{proof} 

We now prove the upper bound in Theorem~\ref{thm:linkcliquelower} using Proposition~\ref{prop:blowup}.

\begin{prop}\label{Ramseylinkobv}
For graphs $G_1,G_2$, we have $r(L_{G_1},L_{G_2}) \leq r(G_1,G_2)+1$. In particular, $$r(L_{K_s},L_{K_n}) \leq r(K_s,K_n)+1 \leq {s+n-2 \choose s-1}+1.$$
\end{prop}
\begin{proof}
The link of a vertex $v$ of a $3$-graph on $N=r(G_1,G_2)+1$ vertices is a graph on $r(G_1,G_2)$ vertices and hence contains a copy of $G_1$ or its complement contains a copy of $G_2$, which together with $v$ forms a copy of $L_{G_1}$ in the $3$-graph or $L_{G_2}$ in the complement of the $3$-graph. 
\end{proof}

From Propositions \ref{Ramseylinkobv} and \ref{prop:blowup}, we have the following immediate corollary. 

\begin{cor}\label{cor:uppblowup}
If $L(m,n)$ denotes the blowup $L_{K_n}(v, m)$ where $v$ is the distinguished vertex of $L_{K_n}$, then
\[
r(L_{K_s},L(m,n)) \leq m \cdot 2^{n+1} \cdot \left({s+n-2 \choose s-1}+1\right)^{n}.
\]
\end{cor}

Observe that any $3$-graph on $n$ vertices which is a subgraph of a blow-up of a link hypergraph is a subgraph of $L(n,n)$. Thus, the upper bound in Theorem~\ref{thm:linkcliquelower} (and the comment immediately afterwards) follows from Corollary~\ref{cor:uppblowup}.

For the rest of this section, we study a quite general problem on hypergraph Ramsey numbers: characterize $k$-graphs $H$ for which the off-diagonal Ramsey number $r(H,K_n^{(k)})$ has a certain growth rate. For example, for which $H$ do these Ramsey numbers grow polynomially in $n$? Of course, the function $r(H,K_n)$ grows polynomially in $n$ for every graph $H$. However, already for $3$-graphs there is no known characterization. 

Let $\Fpoly$ be the family of $3$-graphs $H$ for which there exists a constant $c(H)$ such that $r(H,K_n^{(3)}) \leq n^{c(H)}$ for all $n$. Proposition~\ref{prop:replace} allows us to build members of $\Fpoly$ recursively.

\begin{prop}
The family $\Fpoly$ contains the single edge $3$-graph $K_{3}^{(3)}$, is closed under taking subgraphs, and if $H,F \in \Fpoly$ then $H(v,F) \in \Fpoly$ for every $v\in V(H)$.
\end{prop}

These facts are already observed in \cite{CoFoSu,ErHa}. 

In the other direction, it follows from an argument of Erd\H{o}s and Hajnal \cite{ErHa} (see also \cite{CoFoSu}) that if $H$ is a $3$-graph on $s$ vertices and there is no edge-coloring $C$ of the complete graph on $V(H)$ with colors $I,II,III$ and vertex ordering of $H$ so that for each edge $\{u,v,w\}$ of $H$ with $u<v<w$, $(u,v)$ is color $I$, $(v,w)$ is color $II$, and $(u,w)$ is color $III$, then $r(H,K_n^{(3)}) \geq 2^{cn}$ for some absolute constant $c$.  We have shown $H$ is not in $\Fpoly$. For all other $3$-graphs besides those discussed above is it not known whether or not they lie in $\Fpoly$. 

We remark that the $3$-graphs $H$ obtained in the preceding paragraph from ``ordered rainbow triangles'' naturally arise not only in hypergraph Ramsey problems, but also in hypergraph Tur\' an problems, see the recent work of Reiher, R\"odl, and Schacht~\cite{ReRoSc}.

Let $\Fpolyfact$ be the family of $3$-graphs $H$ for which $r(H,K_n^{(3)}) = n^{O(n)}$ for all $n$ (the implied constant may depend on $H$). In other words, $\Fpolyfact$ consists of those $3$-graphs for which the off-diagonal Ramsey number grows at most polynomially in $n!$. The family $\Fpolyfact$ is closed under the same blowup operation as $\Fpoly$, and by Theorem~\ref{thm:linkcliqueversusclique}, link hypergraphs lie in it as well.

\begin{prop}\label{prop:polyfact}
The family $\Fpolyfact$ contains all link $3$-graphs $L_G$, is closed under taking subgraphs, and if $H,F \in \Fpolyfact$ then $H(v,F) \in \Fpolyfact$ for every $v\in V(H)$.
\end{prop}

We will use Proposition~\ref{prop:polyfact} to prove Theorem~\ref{thm:constant-fraction}.
\begin{proof}[Proof of Theorem~\ref{thm:constant-fraction}.]
Let $E(s)$ be the maximum number of edges of a $3$-graph on $s$ vertices in $\Fpolyfact$. So there is a $3$-graph $F$ in $\Fpolyfact$ on $s$ vertices with $E(s)$ edges. If $t \le E(s)$ and $H$ is a $3$-graph such that among any $s$ vertices of $H$, there are fewer than $t$ edges, then $H$ is $F$-free and hence the independence number of $H$ is $\Omega(\log N/\log \log N)$. Thus, $f_k(N;s,t) \ge \Omega(\log N/ \log \log N)$ for all $t\le E(s)$.

Note that $L_{K_{s_2}} \in \Fpolyfact$ for any positive integer $s_2$ by Theorem \ref{thm:linkcliqueversusclique}. By Proposition~\ref{prop:polyfact}, whenever $s_1+s_2=s$ and $F\in \Fpolyfact$ has $s_1$ vertices, the $3$-graph $L_{K_{s_2}}(v,F)$ is also in $\Fpolyfact$ for any $v\in V(L_{K_{s_2}})$. Taking $v$ to be the distinguished vertex in $L_{K_{s_2}}$ and $F$ to have $E(s_1)$ edges, this gives
\begin{equation}\label{eq:recursivelower} E(s_1+s_2) \geq E(s_1)+s_1{s_2 \choose 2}.
\end{equation}

Using (\ref{eq:recursivelower}) with $s_1 = \alpha s$, $s_2 = (1-\alpha)s$ with $\alpha = \frac{\sqrt{3}}{2}-\frac{1}{2}$ and iterating the recurrence, we see that $\lim_{s \to \infty} E(s)/s^3 \geq \beta$,\footnote{That this limit exists follows easily from Proposition \ref{prop:polyfact}.} where $\beta$ satisfies
\[
\beta = \beta \alpha^3 + \alpha (1-\alpha)^2/2.
\]
It follows that $E(s)\ge (\beta - o(1))s^3$. Since $6\beta > 0.464$, this proves that for all $t\le (0.464 -o(1)){s\choose 3}$, we have $f_3(N;s,t)  = \Omega(\log N/ \log \log N)$.

Let $e(s)$ be the maximum number of edges in a $3$-graph on $s$ vertices with the property that the link of each vertex is bipartite. Erd\H{o}s and S\'os conjectured (see \cite{FF84,MPS11}, page 238),  that $e(s)=(\frac{1}{4}+o(1)){s \choose 3}$, and the current best known upper bound is $e(s)\leq (0.266+o(1)){s \choose 3}$ due to Razborov \cite{Razborov10} using the flag algebra method. 

By Theorem~\ref{thm:oddgirthlowerbound}, for any fixed $s$, there exists a $3$-graph $\Gamma$ on $N$ vertices with $\alpha(\Gamma) = O(\log N/ \log \log N)$ so that the link of each vertex has odd girth greater than $s$. In particular, if $U$ is any set of $s$ vertices of $\Gamma$, the induced subhypergraph $\Gamma[U]$ has bipartite links. By the definition of $e(s)$, this means that among any $s$ vertices of $\Gamma$, there are at most $e(s)$ edges. This proves that $f_3(N;s,t) = O(\log N/\log \log N)$ whenever $t>e(s)$, as desired.
\end{proof}

\section{Closing Remarks} \label{sec:closing}

There are many interesting open problems on hypergraph Ramsey numbers of $3$-uniform hypergraphs. We discuss here a few particularly relevant questions.

Our first problem is about the dependence of the implicit constants in Theorem~\ref{thm:linkhypergraph} on the graph $G$. As $G$ varies through non-bipartite graphs, the explicit upper and lower bounds in  Theorem~\ref{thm:linkhypergraph} we obtain from Theorems~\ref{thm:oddgirthlowerbound} and~\ref{thm:linkcliqueversusclique} are of the form
\[
n^{\Omega(n/g)} \le r(L_G, K_n^{(3)}) \le n^{O(sn)},
\]
where $g$ is the odd girth of $G$ and $s$ is the number of vertices of $G$. We make the following conjecture.
 
\begin{conj}
For every non-bipartite graph $G$, there exists a constant $c_G$ for which
\[
r(L_G, K_n^{(3)}) = n^{(1+o(1))c_Gn}.
\]
\end{conj}

\noindent It is not hard to show that if the constants $c_G$ exist for odd cycles $G=C_g$, then they must be bounded above by an absolute constant. However, we do not know if $c_G$ (if it exists) goes to zero as $G$ ranges through all odd cycles $C_{g}$.

In Section~\ref{sec:recursive}, we defined the families $\Fpoly$ and $\Fpolyfact$ of $3$-graphs $H$ for which $r(H,K_n^{(3)}) \le n^{O(1)}$ and $r(H,K_n^{(3)}) \le n^{O(n)}$, respectively. We also showed that these families are closed under blow-ups, giving a recursive procedure for constructing large subfamilies of both. We conjecture that there are $3$-graphs not in $\Fpolyfact$. 

It is natural to also study the families $\Fpartpoly$ and $\Fpartpolyfact$  of $3$-graphs $H$ for which $r(H,K_{n,n,n}^{(3)})  \leq n^{O(1)}$ and $r(H,K_{n,n,n}^{(3)})  \leq n^{O(n)}$, respectively. Since $K_{n,n,n}^{(3)}\subset K_{3n}^{(3)}$, we have $\Fpoly \subseteq \Fpartpoly$ and $\Fpolyfact \subseteq \Fpartpolyfact$. Furthermore, it is easy to check that $\Fpartpoly$ and $\Fpartpolyfact$ are closed under blowups, i.e.\ if $H, F\in \Fpartpoly$ and $v\in V(H)$, then $H(v,F)\in\Fpartpoly$ as well, and similarly for $\Fpartpolyfact$.

We now show that these two new families are closed under an additional operation.

\begin{prop}\label{prop:indep-edge}
Let $H$ be a fixed $3$-graph on $h\ge 3$ vertices and $\{u,v,w\}$ be a triple of vertices of $H$ no pair of which lie in an edge together, and let $H'$ be the $3$-graph obtained by adding the edge $\{u,v,w\}$ to $H$. Then, $r(H',K_{n,n,n}^{(3)}) \le 3n\cdot r(H, K_{n,n,n}^{(3)})^h$.
\end{prop}
\begin{proof}
Let $N = r(H, K_{n,n,n}^{(3)})$, and suppose for contradiction that there exists an $H'$-free $3$-graph $\Gamma$ on $M = 3n \cdot N^h$ vertices whose complement contains no $K_{n,n,n}^{(3)}$. Let $H\setminus\{u,v,w\}$ be the $3$-graph obtained by removing the vertices $u,v,w$ from $H$. For each embedding $\phi:H\setminus\{u,v,w\}\hookrightarrow\Gamma$, we let $U(\phi)$ be the set of vertices $u_0\in V(\Gamma)$ such that if we extend $\phi$ by setting $\phi(u)=u_0$, $\phi$ is an embedding of $H\setminus\{v,w\}$. Define $V(\phi)$, $W(\phi)$ similarly.

Since there are no edges of $H$ containing any pair of $\{u,v,w\}$, it follows that if $u_0 \in U(\phi)$, $v_0 \in V(\phi)$ and $w_0 \in W(\phi)$ are distinct, then these three vertices form a copy of $H$ together with $\textup{im}(\phi)$. Since $\Gamma$ is $H'$-free, no such triple $(u_0,v_0,w_0)\in U(\phi)\times V(\phi) \times W(\phi)$ of distinct vertices can be an edge of $\Gamma$. It remains to show that for some choice of $\phi$, $|U(\phi)|$, $|V(\phi)|$, and $|W(\phi)|$ are all at least $3n$, since this would then imply the existence of a copy of $K_{n,n,n}^{(3)}$ in the complement of $\Gamma$.

This is a standard double-counting argument. Suppose for the sake of contradiction that for any choice of $\phi:H\setminus\{u,v,w\}\hookrightarrow\Gamma$, at least one of $|U(\phi)|$, $|V(\phi)|$, and $|W(\phi)|$ is smaller than $3n$. For each such $\phi$,
\[
|U(\phi)|\cdot |V(\phi)|\cdot |W(\phi)| < 3n \cdot M^2.
\]
Because every copy of $H$ in $\Gamma$ can be obtained from some $\phi$, and there are at most $M^{h-3}$ choices of $\phi$, the total number of copies of $H$ in $\Gamma$ is less than $3n M^{h-1}$.

On the other hand, we know that among any $N= r(H, K_{n,n,n}^{(3)})$ vertices of $\Gamma$, there is a copy of $H$, since $\overline{\Gamma}$ is $K_{n,n,n}^{(3)}$-free. Each single copy of $H$ lies in at most ${M-h \choose N-h}$ of the $N$-subsets of $\Gamma$, so the total number of copies of $H$ in $\Gamma$ is at least
\[
\frac{{M\choose N}}{{M-h\choose N-h}} = \frac{(M)_h}{(N)_h} \ge \Big(\frac{M}{N}\Big)^h \ge 3n M^{h-1},
\]
where $(x)_h = x(x-1)\cdots(x-h+1)$ is the falling factorial. This contradicts our previous conclusion that there are fewer than $3n M^{h-1}$ copies of $H$, so we are done.
\end{proof}

Note that the previous proposition can be iterated to show that every linear $3$-graph lies in $\Fpartpoly$. We suspect that $\Fpartpoly$ is significantly larger than $\Fpoly$, and similarly $\Fpartpolyfact$ is larger than $\Fpolyfact$. We make the following quantitative conjecture.

\begin{conj}
If
\[
e(\F) = \limsup_{H\in \F} \frac{e(H)}{{v(H)\choose 3}},
\]
then $e(\Fpoly) < e(\Fpartpoly)$ and $e(\Fpolyfact)<e(\Fpartpolyfact)$.
\end{conj}

\noindent {\bf Acknowledgements.} The first author would like to thank David Conlon, Nina Kamcev, Benny Sudakov, and Fan Wei for early discussions related to this paper. We are also grateful to Yuval Wigderson and the referee for many helpful comments on this manuscript.

\appendix

\section{Calculation for Theorem~\ref{thm:linkcliquelower}}\label{app:threshold}

\begin{lem}
For $s\ge 14$, $n \ge 2$, $m = {n+s \choose s}^{2/13}$, and $p = m^{-\frac{2}{s-1}}$,
\begin{equation}\label{eq:M-bound}
    \min \Big\{ (1-p^3)^{-3n}, \frac{m}{486en^2}\Big\} \ge {n + s \choose s}^{\Omega(1)}.
\end{equation}
\end{lem}
\begin{proof}
Note that for $s\ge 14$, we have that
\[
m^{\frac{13}{14}} = {n+s \choose s}^{1/7} \ge \Omega(n^2).
\]
Thus,
\[
\frac{m}{486en^2} \ge \Omega(m^{\frac{1}{14}}) \ge {n + s \choose s}^{\Omega(1)},
\]
and it remains to show
\begin{equation}\label{eq:p-bound}
(1-p^3)^{-3n} \ge {n + s \choose s}^{\Omega(1)}.
\end{equation}
Now we break into two cases depending on whether or not $s\le 10n$. If $s\le 10n$, we use the bound $1-p^3 \le e^{-p^3}$, which gives
\[
(1-p^3)^{-3n} \ge e^{3p^3n} = e^{3m^{-6/(s-1)}n}.
\]
Taking logarithms, it suffices to show that
\[
m^{-6/(s-1)}n \ge \Omega\Big(\log {n + s \choose s}\Big).
\]
By our choice of $m$ and the fact that ${n + s \choose s} \le (11en/s)^s$ if $s\le 10n$, it suffices to show that
\[
\Big(\frac{11en}{s}\Big)^{\frac{12s}{13(s-1)}} \cdot s \log(11en/s) = O(n)
\]
uniformly over all $14\le s\le 11n$. The exponent $\frac{12s}{13(s-1)}$ takes its maximum value $1-\frac{1}{169}$ when $s=14$, so we get
\[
\Big(\frac{11en}{s}\Big)^{\frac{12s}{13(s-1)}} \cdot s \log(11en/s) \le (11en)^{1 - \frac{1}{169}} s^{\frac{1}{169}} \log (11en/s).
\]
Finally, using the fact that $\log(11en/s) = O((11en/s)^{1/169})$, the desired expression is $O(n)$ uniformly in $14\le s \le 10n$, as desired. This proves (\ref{eq:p-bound}) for $s\le 10n$.

Now we consider the case $s>10n$. In this case, $p = m^{-2/(s-1)}$ will be very close to $1$, so we write
\[
1-p^3 = 1-m^{-6/(s-1)} = 1-e^{-6\log m /(s-1)} \le \frac{6\log m}{s-1},
\]
using the fact that $1-e^{-x} \le x$ for all real $x$. For $s>10n$, we have ${n+s \choose s} \le (2es/n)^n$, so
\[
6\log m \le \frac{12}{13} n \log (2es/n) \le n \log (2es/n) \le \sqrt{2esn},
\]
since for $x> 20e$, $\log x \le \sqrt{x}$. Thus,
\[
(1-p^3)^{-3n} \ge \Big(\frac{s-1}{\sqrt{2esn}}\Big)^{3n} \ge \Big(\frac{2es}{n}\Big)^{\Omega(n)} \ge {n + s \choose s}^{\Omega(1)}
\]
as desired. This completes the proof of (\ref{eq:p-bound}) and the theorem.
\end{proof}

\end{document}